\documentclass{amsart}

\usepackage[foot]{amsaddr}
\let\oldtocsection=\tocsection

\let\oldtocsubsection=\tocsubsection 

\let\oldtocsubsubsection=\tocsubsubsection

\renewcommand{\tocsection}[2]{\vspace{0.5em}\hspace{0em}\oldtocsection{#1}{#2}}
\renewcommand{\tocsubsection}[2]{\vspace{0.5em}\hspace{1em}\oldtocsubsection{#1}{#2}}
\renewcommand{\tocsubsubsection}[2]{\vspace{0.5em}\hspace{2em}\oldtocsubsubsection{#1}{#2}}
\usepackage{graphicx,cancel,xcolor,hyperref,comment,graphicx,geometry}

\usepackage{tikz}
\usetikzlibrary{decorations.pathreplacing}
\usepackage{graphicx,url,etoolbox}
\usepackage{lipsum}
\usepackage{dsfont}
\usepackage{amsmath,amssymb}

\usepackage{verbatim}

\usepackage{fancyhdr}
\usepackage{lastpage}

\newtheorem{theoreme}{Theorem}[section]

\theoremstyle{definition}

\numberwithin{equation}{section}

\renewenvironment{proof}{{\bfseries \noindent Proof.}}{\demo}
\newcommand\xqed[1]{%
	\leavevmode\unskip\penalty9999 \hbox{}\nobreak\hfill
	\quad\hbox{#1}}
\newcommand\demo{\xqed{$\square$}}

\hypersetup{bookmarks, colorlinks, urlcolor=blue, citecolor=blue, linkcolor=blue, hyperfigures, pagebackref,
	pdfcreator=LaTeX, breaklinks=true, pdfpagelayout=SinglePage, bookmarksopen=true,bookmarksopenlevel=2}

\usepackage{comment}

\def\R{\mathbb R}
\def\Z{\mathbb Z}
\def\N{\mathbb N}

\def\HH{\mathcal H}

\def\AA{\mathcal A}

\def\la {{\lambda}}

\newcommand {\nc}   {\newcommand}
\nc {\be}   {\begin{equation}} \nc {\ee}   {\end{equation}} \nc
{\beq}  {\begin{eqnarray}} \nc {\eeq}  {\end{eqnarray}} \nc {\beqs}
{\begin{eqnarray*}} \nc {\eeqs} {\end{eqnarray*}}
\def\edc{\end{document}}

\providecommand{\abs}[1]{\lvert#1\rvert}

\numberwithin{equation}{section}

\theoremstyle{Thm}
\newtheorem{Thm}{Theorem}[section]
\newtheorem{lem}{Lemma}[section]
\newtheorem{prop}{Proposition}[section]

\newtheorem{rk}{Remark}[section]
\definecolor{carnelian}{rgb}{0.7, 0.11, 0.11}
\definecolor{carmine}{rgb}{0.59, 0.0, 0.09}
\definecolor{burgundy}{rgb}{0.5, 0.0, 0.13}
\definecolor{darkmidnightblue}{rgb}{0.0, 0.2, 0.4}
\definecolor{dimgray}{rgb}{0.75, 0.75, 0.75}
\definecolor{palecarmine}{rgb}{0.69, 0.25, 0.21}
\newcounter{dummy} 
\numberwithin{dummy}{section}
\newtheorem{Theorem}[dummy]{Theorem}

\newtheorem{defi}[dummy]{Definition}

\numberwithin{equation}{section}

\def\AA{\mathcal A}
\def\HH{\mathbf{\mathcal H}}

\newcommand{\h}{\mathsf{h}}
\newcommand{\f}{\mathsf{f}}
\newcommand{\s}{\mathsf{S}}
\newcommand{\g}{\mathsf{f}_3}
\newcommand{\q}{\mathsf{q}}
\newcommand{\ve}{\mathsf{v}}
\newcommand{\intdx}{\int_{0}^{L} }

\newcommand{\intdnd}{\int_{\alpha }^{\beta } }

\providecommand{\abs}[1]{\lvert#1\rvert}
\begin{document}
	\title[\fontsize{7}{9}\selectfont  ]{On the Stability of  Bresse system with one discontinuous local internal Kelvin-Voigt damping  on the axial force}
\author{Mohammad Akil$^{1}$}
\author{Haidar Badawi$^{2}$}
\author{Serge Nicaise$^{2}$}
\author{Ali Wehbe$^{3}$}
\address{$^1$ Universit\'e Savoie Mont Blanc, Laboratoire LAMA, Chamb\'ery-France}
\address{$^2$ Universit\'e Polytechnique Hauts-de-France (UPHF-LAMAV),
	Valenciennes, France}
\address{$^3$Lebanese University, Faculty of sciences 1, Khawarizmi Laboratory of  Mathematics and Applications-KALMA, Hadath-Beirut, Lebanon.}
\email{mohammad.akil@univ-smb.fr, Haidar.Badawi@etu.uphf.fr, Serge.Nicaise@uphf.fr,   ali.wehbe@ul.edu.lb}
\begin{abstract}
	In this paper, we investigate the stabilization of a linear Bresse system with one discontinuous local internal viscoelastic damping of Kelvin-Voigt type acting on the axial force, under fully Dirichlet boundary conditions. First, using a general criteria of Arendt-Batty,  we prove the strong stability of our system. Finally, using a frequency domain approach combined with the multiplier method,  we prove that the energy of our system  decays polynomially with different rates.
\end{abstract}
\maketitle
\pagenumbering{roman}
\maketitle
\tableofcontents
\pagenumbering{arabic}
\setcounter{page}{1}
\newpage
\section{Introduction}
\noindent In this paper, we investigate the stability of Bresse system with only one discontinuous local internal Kelvin-Voigt damping on the axial force. More precisely, we consider the following system:
\begin{equation}\label{p3-sysorig}
\left\{	\begin{array}{llll}
	\displaystyle \rho_1 \varphi_{tt}-k_1 (\varphi_x+\psi+lw)_x -lk_3 (w_x-l\varphi)-ld(x)(w_{tx}-l\varphi_t )=0, &(x,t)\in  (0,L) \times (0,\infty),&\vspace{0.15cm}\\
	\displaystyle \rho_2 \psi_{tt}-k_2 \psi_{xx} +k_1 (\varphi_x +\psi+lw)=0,&(x,t)\in  (0,L) \times (0,\infty),&\vspace{0.15cm}\\
\displaystyle \rho_1 w_{tt}-\left[k_3 (w_x-l\varphi)+d(x)(w_{tx}-l\varphi_t) \right]_x+lk_1 (\varphi_x+\psi+lw)=0, &(x,t)\in  (0,L) \times (0,\infty),
	\end{array}
	\right.
\end{equation}
with the following Dirichlet boundary conditions 
\begin{equation}
	\varphi(0,t)=\varphi(L,t)=\psi(0,t)=\psi(L,t)=w (0,t)=w(L,t)=0, \ \ t>0,
\end{equation}
and the following initial conditions
\begin{equation}\label{p3-initialcond}
\left\{\begin{array}{lll}
\displaystyle 	\varphi(x,0)=\varphi_0(x), \ \varphi_t (x,0)=\varphi_1(x), \ \psi(x,0)=\psi_0(x),  \ x\in(0,L),\vspace{0.15cm}\\
\displaystyle  \psi_t(x,0)=\psi_1(x), \ w(x,0)=w_0(x), \ w_t (x,0)=w_1(x), \ x\in(0,L),
	\end{array}
	\right.
\end{equation}
where $\rho_1, \rho_2, k_1, k_2, k_3, l $ and $L$ are  positive real numbers. We suppose that there exists $0<\alpha<\beta<L$ and a positive constant $d_0$ such that 
\begin{equation}\label{p3-a}
	d(x)=\left\{\begin{array}{lll}
	d_0 & \text{if} & x\in (\alpha,\beta),\vspace{0.15cm}\\
		0 & \text{if} & x\in (0,\alpha)\cup(\beta,L).
	\end{array}
	\right.
\end{equation}
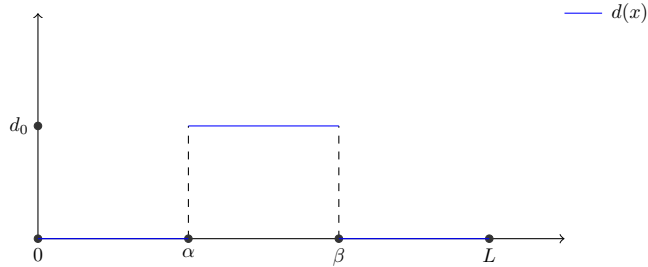
\begin{figure}[h!]
	\begin{center}
		\begin{tikzpicture}
		\draw[->](0,0)--(7,0);
		\draw[->](0,0)--(0,3);

		\draw[dashed](2,0)--(2,1.5);
		\draw[dashed](4,0)--(4,1.5);

		\node[black,below] at (2,0){\scalebox{0.75}{$\alpha$}};
		\node at (2,0) [circle, scale=0.3, draw=black!80,fill=black!80] {};

		\node[black,below] at (4,0){\scalebox{0.75}{$\beta $}};
		\node at (4,0) [circle, scale=0.3, draw=black!80,fill=black!80] {};
		
		\node[black,below] at (6,0){\scalebox{0.75}{$L $}};
		\node at (6,0) [circle, scale=0.3, draw=black!80,fill=black!80] {};

		
		\node[black,below] at (0,0){\scalebox{0.75}{$0$}};
		\node at (0,0) [circle, scale=0.3, draw=black!80,fill=black!80] {};
		
		\node at (0,1.5) [circle, scale=0.3, draw=black!80,fill=black!80] {};
		\node[black,left] at (0,1.5){\scalebox{0.75}{$d_0$}};

		\node[black,right] at (7.5,3){\scalebox{0.75}{$d(x)$}};

		\draw[-,blue](7,3)--(7.5,3);
		\draw[-,blue](0,0)--(2,0);
		\draw[-,blue](2,1.5)--(4,1.5);
		\draw[-,blue](4,0)--(6,0);

		\end{tikzpicture}
	\end{center}
	\caption{Geometric description of the function $d(x)$.}\label{p3-Fig1}
\end{figure}\\
The Bresse system  is a model for arched beams, see \cite[Chap. 6]{LagneseLeugering01}. It can be expressed by the equations of motion:
\begin{equation}\label{p3-1.5}
\left\{\begin{array}{lll}
\displaystyle	\rho_1 \varphi_{tt}=Q_x +lN,\vspace{0.15cm}\\
\displaystyle	\rho_2 \psi_{tt}=M_x -Q,\vspace{0.15cm}\\
\displaystyle	\rho_1 w_{tt}=N_x -lQ,
\end{array}
\right.
\end{equation}
where  $N=k_3 (w_x-l\varphi)+d(x)(w_{tx}-l\varphi_t)$ is the axial force, $Q=k_1 (\varphi_x+\psi+lw)$ is the shear force, and $M=k_2 \psi_{x}$ is the bending moment.
The functions $\varphi$, $\psi$, and $w$ are respectively the vertical, shear angle, and longitudinal displacements.
Here $\rho_1 =\rho A$, $\rho_2 =\rho I$, $k_1 =kGA$, $k_3 =EA$, $k_2 =EI$ and $l=R^{-1}$, in which  $\rho$ is the density of the material, 
$E$ the modulus of the elasticity, $G$ the shear modulus, $k$ the shear factor,  $A$ the cross-sectional area,  $I$ the second moment of area of the cross section, $R$ the radius of the curvature, and $l$ the curvature.\\ \linebreak  There are several publications concerning the stabilization of Bresse system with different kinds of damping (see \cite{Wehbe08}, \cite{doi:10.1002/mma.6070}, \cite{ALABAUBOUSSOUIRA2011481}, \cite{deLima2018}, \cite{ElArwadi2019}, \cite{doi:10.1080/00036811.2018.1520982}, \cite{FATORI2012600}, \cite{10.1093/imamat/hxq038}, \cite{Guesmia2017}, \cite{doi:10.1002/mma.3228}  \cite{RaoLiu03}, \cite{Wehbe03}, \cite{Wehbe02} and \cite{Wehbe01}). We note that by neglecting $w$ ($l \to 0$) in \eqref{p3-1.5}, the Bresse system  reduces to the following conservative Timoshenko system: 
\begin{equation*}
\begin{array}{lll}
\rho_1 \varphi_{tt}-k_1 (\varphi_x+\psi)_x=0,\vspace{0.25cm}\\
\rho_2 \psi_{tt}-k_2 \psi_{xx}+k_1 (\varphi_x+\psi)=0.
\end{array}
\end{equation*}
There are also several publications concerning the stabilization of Timoshenko system  with different kinds of damping (see  \cite{Akil2020}, \cite{BASSAM20151177}, \cite{doi:10.1002/zamm.201500172} and  \cite{doi:10.1080/00036810903156149}).\\ \linebreak
In the recent years, many researchers showed interest in problems involving  Kelvin-Voigt  damping where different types of stability, depending on the smoothness of the damping coefficients,  has been showed (see \cite{Alves2014}, \cite{Alves2013}, \cite{F.HASSINE2015}, \cite{HASSINE201584}, \cite{Huang-1988}, \cite{chenLiuLiu-1998},  \cite{Liu2016},  \cite{doi:10.1137/15M1049385}, \cite{Portillo2017} and \cite{rivera2018}). Moreover, there is a number of new results concerning systems with local Kelvin-Voigt damping and non-smooth coefficients at the interface (see \cite{akilbadawiwehbe2020stability}, \cite{akilissawehbe}, \cite{doi:10.1002/mma.6918}, \cite{ghader2020stability}, \cite{ghader2020transmission2}, \cite{Hayek} and \cite{NASSER2019272}).\\ \linebreak 
Among this vast literature let us recall some  specific results on the Bresse systems.\\ \linebreak
\indent In 2017, Guesmia in \cite{Guesmia2017} studied the stability of Bresse system with one infinite memory in the longitudinal displacement (i.e. third equation) under Dirichlet-Neumann-Neumann boundary conditions, he established some stability 
results provided that the curvature $l$ and the memory kernel $g$ satisfy:
$$l<\tilde{l}, \ \ \text{for some}\ \  \tilde{l}>0 \quad \text{and} \quad g^0:=\int_{0}^{\infty}g(s)ds<\tilde{g}, \ \ \text{for some}\ \  \tilde{g}>0. $$
In 2018, Afilal {\it et al.} in \cite{doi:10.1002/mma.6070} studied the stability of Bresse system with global frictional damping in the longitudinal displacement, by considering the following system on $(0,1)\times (0,\infty)$:
\begin{equation}\label{p3-1.6}
\left\{	\begin{array}{llll}
\displaystyle \rho_1 \varphi_{tt}-k_1 (\varphi_x+\psi+lw)_x -lk_3 (w_x-l\varphi)=0, \vspace{0.15cm}\\
\displaystyle \rho_2 \psi_{tt}-k_2 \psi_{xx} +k_1 (\varphi_x +\psi+lw)=0,\vspace{0.15cm}\\
\displaystyle\rho_1 w_{tt}-k_3 (w_x-l\varphi)+lk_1 (\varphi_x+\psi+lw)+\delta w_t=0, 
\end{array}
\right.
\end{equation}
with the initial conditions \eqref{p3-initialcond} where $L=1$ and under mixed boundary conditions of the form:
$$
\left\{ \begin{array}{lll}
\displaystyle \varphi(0,t)=\psi_x(0,t)=w_x (0,t)=0, \quad \text{in} \ \ (0,\infty),\vspace{0.15cm}\\  \displaystyle \varphi_x(1,t)=\psi(1,t)=w(1,t)=0,\ \, \quad \text{in} \ \ (0,\infty),
\end{array}
\right.
$$ 
where $\delta$ is a positive real number, they assumed that:
\begin{equation}\label{p3-1.7*}
l\neq \frac{\pi}{2}+m\pi, \quad \forall \,  m\in \N.
\end{equation} 
They proved under \eqref{p3-1.7*} the strong stability of  system \eqref{p3-1.6} provided that the curvature $l$ satisfies:
\begin{equation}\label{p3-1.8*}
l^2\neq \frac{\rho_2k_3+\rho_1k_2}{\rho_2 k_3}\left(\frac{\pi}{2}+m\pi\right)^2+\frac{\rho_1k_1}{\rho_2(k_1+k_3)}, \ \ \forall m\in \Z.
\end{equation}
Also, they established under \eqref{p3-1.7*} and \eqref{p3-1.8*} the exponential stability of system \eqref{p3-1.6} if and only if  $\dfrac{k_1}{\rho_1}=\dfrac{k_2}{\rho_2}=\dfrac{k_3}{\rho_1}$. Otherwise, they established polynomial energy decay rate of order   $t^{-\frac{1}{4}}$. In 2019, Fatori {\it et al.} in \cite{doi:10.1080/00036811.2018.1520982} proved under 
\begin{equation}\label{p3-1.7**}
lL \quad \text{is not a multiple of} \quad \pi,
\end{equation}
the strong stability of system \eqref{p3-1.6} on $(0,L)\times (0,\infty)$ under Dirichlet-Neumann-Neumann boundary conditions provided that:
\begin{equation}\label{p3-1.7}
k_1\rho_1 -\rho_2 (k_3+k_1)l^2\geq 0 \quad \text{or} \quad 0< \rho_2 (k_3+k_1)l^2 -k_1\rho_1 \neq \frac{\rho_1\rho_2(k_3+k_1)}{k_3}\left(\frac{k_3}{\rho_1}n^2+\frac{k_2}{\rho_2}m^2 \right)\frac{\pi^2}{L^2},
\end{equation}
for all $m \in \N$ and $n \in \N^{\star}$. Also, they established under \eqref{p3-1.7**} and \eqref{p3-1.7} the exponential stability  of system \eqref{p3-1.6} on $(0,L)\times (0,\infty)$ if and only if \begin{equation}\label{p3-1.8}
\dfrac{\rho_1}{\rho_2}=\dfrac{k_1}{k_2} \quad \text{and}\quad  k_1=k_3.
\end{equation}
Moreover, they used the previous results (i.e. strong and exponential stability of \eqref{p3-1.6} on $(0,L)\times (0,\infty)$) to obtain under \eqref{p3-1.7**}, \eqref{p3-1.7} and \eqref{p3-1.8} the exponential stability of Bresse system with indefinite memory in the longitudinal displacement under Dirichlet-Neumann-Neumann boundary conditions.\\ \linebreak
 \indent In 2019, El Arwadi and Youssef in \cite{ElArwadi2019} studied the stabilization of the Bresse beam with three global Kelvin-Voigt damping under fully Dirichlet boundary conditions, they established an exponential energy decay rate. In 2020, Gerbi {\it et al.} in \cite{gerbi2020stabilization} studied the stabilization of non-smooth transmission problem involving Bresse systems with fully Dirichlet or Dirichlet-Neumann-Neumann boundary conditions, by considering system \eqref{p3-1.5} on $(0,L)\times (0,\infty)$ with
$$
N=k_3 (w_x-l\varphi)+D_3 (w_{xt}-l\varphi_t), \ \ Q=k_1 (\varphi_x+\psi+lw)+D_1 (\varphi_{xt}+\psi_t+lw_t ), \ \ M=k_2 \psi_x +D_2 \psi_{xt},
$$
where $D_1$, $D_2$ and $D_3 $ are bounded positive functions over $(0,L)$.
 They established: 
\begin{itemize}
	\item Analytic stability in the case of three global Kelvin-Voigt dampings  (i.e. $D_i \in L^\infty (0,L)$, $D_i \geq d_0 >0 \ \text{in}\ (0,L)$, $i=1,2,3$). \\ 
	\item Exponential stability in the case of three local Kelvin-Voigt dampings with smooth coefficients at the interface  (i.e. $D_i \in W^{1,\infty}(0,L)$, $D_i \geq d_0 >0 \ \text{in}\ \emptyset\neq \omega:=(\alpha,\beta)\subset (0,L)$, $i=1,2,3$).\\
	\item Polynomial energy decay rate of order $t^{-1}$ in the case of  three local Kelvin-Voigt dampings with non-smooth coefficients at the interface  (i.e. $D_i \in L^\infty (0,L)$, $D_i \geq d_0^i >0 \ \text{in}\ (\alpha_i,\beta_i)\subset (0,L)$, $i=1,2,3$, and $\displaystyle \bigcap_{i=1}^3 (\alpha_i,\beta_i)=\omega $).\\
	\item  Polynomial stability energy decay rate  of order $t^{-\frac{1}{2}}$ in the case of one local Kelvin-Voigt damping on the bending moment with non-smooth coefficient at the interface (i.e.  $D_1=D_3=0 $, $D_2\in L^\infty (0,L)$ and $D_2 \geq d_0 >0 $ in $\omega$).
\end{itemize}.\\
  But to the best of our knowledge, it seems that  no result in the literature exists concerning the case of Bresse system with only one discontinuous local internal Kelvin-Voigt damping on the axial force, especially under fully Dirichlet boundary conditions and without any condition on the curvature $l$. The goal of the present paper is to fill this gap by studying the stability of system \eqref{p3-sysorig}-\eqref{p3-initialcond}.\\\linebreak 
  This paper is organized as follows: In Section \ref{p3-WPS}, we prove the well-posedness of our system by using semigroup approach. In Section \ref{p3-sec3}, following a general criteria of Arendt Batty, we show the strong stability of our system in the absence of the compactness of the resolvent. Finally, in Section \ref{p3-secpoly}, by using the frequency domain approach  combining with a specific  multiplier method,  we prove that the energy of our system  decays polynomially with the rates:
  \begin{equation*}
  \left\{	\begin{array}{lll}
  \displaystyle	t^{-1} \quad \text{if} \quad \frac{k_1}{\rho_1}=\frac{k_2}{\rho_2},\vspace{0.15cm}\\
 \displaystyle 		t^{-\frac{1}{2}} \quad \text{if} \quad \frac{k_1}{\rho_1}\neq\frac{k_2}{\rho_2}.
  	\end{array}
  	\right.
  \end{equation*}

\section{Well-posedness of the system}\label{p3-WPS}
\noindent In this section,  we will establish the well-posedness of system \eqref{p3-sysorig}-\eqref{p3-initialcond} by using semigroup approach.
The energy of system \eqref{p3-sysorig}-\eqref{p3-initialcond} is given by 
$$
\begin{array}{lll}
\displaystyle E(t)=\frac{1}{2}\intdx \left(\rho_1 \left|\varphi_t\right|^2+\rho_2|\psi_t|^2 +\rho_1 |w_t|^2 +k_1 |\varphi_x+\psi+lw|^2 +k_2 |\psi_x|^2 +k_3 |w_x -l\varphi|^2\right) dx.
\end{array}  
$$
Let $(\varphi,\varphi_t,\psi,\psi_t,w,w_t)$ be a regular solution of system \eqref{p3-sysorig}-\eqref{p3-initialcond}. Multiplying the equations in \eqref{p3-sysorig} by $\overline{\varphi_t}$, $\overline{\psi_t}$ and $\overline{w_t}$ respectively, Then using the boundary conditions \eqref{p3-bc} and the definition of $d(x)$ (see \eqref{p3-a} and Figure \ref{p3-Fig1}), we obtain  
\begin{equation}\label{p3-2.12}
	E^{\prime }(t)=-\intdx d(x) |w_{tx}-l\varphi_t|^2 dx=-d_0\intdnd  |w_{tx}-l\varphi_t|^2 dx \leq 0.
\end{equation}
From \eqref{p3-2.12}, system \eqref{p3-sysorig}-\eqref{p3-initialcond} is dissipative in the sense that its energy is non-increasing with respect to time. Now, we define the following Hilbert space $\HH$ by:
$$
\HH:=\left( H^1_0(0,L)\times L^2 (0,L)\right)^3.
$$
The Hilbert space $\HH$ is equipped with the following inner product  and norm 
$$
\begin{array}{lll}
\displaystyle (U,U^1 )_{\HH}=\intdx \left\{ k_1(v^{1}_x+v^3+lv^5)\overline{(\widetilde{v^{1}_x}+\widetilde{v^3}+l\widetilde{v^5})}+\rho_1 v^2\overline{\widetilde{v^2}}+k_2 v^3_x \overline{\widetilde{v^3_x}} +\rho_2 v^4\overline{\widetilde{v^4}}\right.\vspace{0.25cm}
\\
\hspace{3cm}\displaystyle \left. +\, k_3 (v^5_x-lv^1)(\overline{\widetilde{v^5_x}-l\widetilde{v^1}})dx+\rho_1 v^6\overline{\widetilde{v^6}}\right\}dx
\end{array}
$$
 and 
\begin{equation}\label{p3-normU}
\|U\|_{\HH}^2 = \intdx \left(k_1|v^1_x+v^3+lv^5|^2  +\rho_1 |v^2|^2 +k_2 |v^3_x|^2 +\rho_2 |v^4|^2 +k_3 |v^5_x -lv^1 |^2 +\rho_1 |v^6|^2 \right) dx. 
\end{equation}
\noindent Where  $U=(v^1,v^2,v^3,v^4,v^5,v^6)^{\top}\in \HH$ and $\widetilde{U} =(\widetilde{v^1}, \widetilde{v^1},\widetilde{v^2},\widetilde{v^3},\widetilde{v^4},\widetilde{v^5}, \widetilde{v^6} )^{\top}\in\HH$. Now, we define the linear unbounded  operator $\AA:D(\AA)\subset \HH\longmapsto \HH$  by:
\begin{equation}
D(\AA)=\left\{\begin{array}{cc}\vspace{0.25cm}
U=(v^1,v^2,v^3,v^4,v^5,v^6)^{\top}\in \HH \,\,|\,\, v^1,v^3 \in H^2 (0,L)\cap H^{1}_0 (0,L) \\\vspace{0.25cm}
\displaystyle v^2,v^4,v^6 \in H_0^1 (0,L), \ \ \left[k_3v^5_x +d(x)(v^6_x-lv^2) \right]_x\in L^2(0,L)
\end{array}\right\} 
\end{equation}and 
\begin{equation}\label{p2-op}
\AA\begin{pmatrix}
v^1\\v^2\\v^3\\v^4\\v^5\\v^6
\end{pmatrix}=
\begin{pmatrix} 
v^2\\\displaystyle  \frac{k_1}{\rho_1}(v^{1}_x+v^3+lv^5 )_x+\frac{lk_3}{\rho_1}(v^{5}_x-lv^1 )+\frac{ld(x)}{\rho_1}(v^6_x-lv^2)\\\displaystyle v^4\\\displaystyle \frac{k_2}{\rho_2}v^{3}_{xx}  -\frac{k_1}{\rho_2}(v^{1}_x+v^3+lv^5 )\\\displaystyle v^6 \vspace{0.15cm}
\\\displaystyle \frac{1}{\rho_1}\left[k_3(v^{5}_x-lv^1 )+d(x)(v^6_x-lv^2) \right]_x-\frac{lk_1}{\rho_1}(v^{1}_x+v^3+lv^5)
\end{pmatrix},
\end{equation} 
for all $U=(v^1,v^2,v^3,v^4,v^5,v^6)^{\top}\in D(\AA)$.\\
In this sequel, $\|\cdot\|$ will denote the usual norm of $L^2 (0,L)$.\\
\begin{rk}\label{p3-rknorm}
{\rm	From Poincar\'e inequality, we deduce that there exists a positive constant $c_1$ such that 
	\begin{equation*}
		k_1 \|v^1_x+v^3+lv^5\|^2 +k_2 \|v^3_x\|^2 +k_3 \|v^5_x-lv^1\|^2 \leq c_1 \left(\|v^1_x\|^2+\|v^3_x\|^2+\|v^5_x\|^2\right),  \ \ \forall (v^1,v^3,v^5)\in \left(H^1_0(0,L)\right)^3.
	\end{equation*}
	Moreover, we can show by a contradiction argument that there exists a positive constant $c_2$ such that
	\begin{equation*}
 c_2 \left(\|v^1_x\|^2+\|v^3_x\|^2+\|v^5_x\|^2\right) \leq k_1 \|v^1_x+v^3+lv^5\|^2 +k_2 \|v^3_x\|^2 +k_3 \|v^5_x-lv^1\|^2 ,  \ \ \forall (v^1,v^3,v^5)\in \left(H^1_0(0,L)\right)^3.
\end{equation*}
Therefore, the norm defined in \eqref{p3-normU} is equivalent to the usual norm of $\HH$.\xqed{$\square$}
}
\end{rk}
\noindent Now, if $U=(\varphi, \varphi_t,\psi,\psi_t,w,w_t)^{\top}$, then system \eqref{p3-sysorig}-\eqref{p3-initialcond} can be written as the following first order evolution equation 
\begin{equation}\label{p3-firstevo}
U_t =\AA U , \quad U(0)=U_0,
\end{equation}
where  $U_0 =(\varphi_0 ,\varphi_1,\psi_0,\psi_1,w_0,w_1 )^{\top}\in \HH$.
\begin{prop}\label{p3-mdissip}
	{\rm The unbounded linear operator $\AA$ is m-dissipative in the Hilbert space $\HH$.}
\end{prop}
\begin{proof}
For all $U=(v^1,v^2,v^3,v^4,v^5,v^6)^\top \in D(\AA)$, we have
\begin{equation}\label{p3-Reau}
	\Re (\AA U,U)_{\HH}=-\intdx d(x)\left|v^6_x-lv^2\right|^2 dx = -d_0 \intdnd \left|v^6_x-lv^2\right|^2 dx \leq 0.
\end{equation}
which implies that $\AA$ is dissipative. Let us prove that $\AA$ is maximal. For this aim, let $F=(f^1,f^2,f^3,f^4,f^5,f^6)^{\top}\in\HH$, we look for $U=(v^1,v^2,v^3,v^4,v^5,v^6)^{\top}\in D(\AA)$ unique solution of 
\begin{equation}\label{p3-AU=F}
	-\AA U=F.
\end{equation}
Detailing \eqref{p3-AU=F}, we obtain 
\begin{eqnarray}
-v^2&=&f^1,\label{p3-f1}\\
-k_1\left(v^1_x+v^3+lv^5 \right)_x -lk_3(v^5_x-lv^1)-ld(x)(v^6_x-lv^2)&=&\rho_1 f^2,\label{p3-f2}\\
-v^4&=&f^3,\label{p3-f3}\\
	-k_2v^{3}_{xx} +k_1(v^1_x+v^3+lv^5)&=&\rho_2 f^4,\label{p3-f4}\\
	-v^6&=&f^5,\label{p3-f5}\\
-\left[k_3\left(v^5_x-lv^1\right)+d(x)(v^6_x-lv^2) \right]_x +lk_1(v^1_x+v^3+lv^5)&=&\rho_1f^6,\label{p3-f6}
\end{eqnarray}
with the following boundary conditions
\begin{equation}\label{p3-bc}
	v^1 (0)=v^1(L)=v^3(0)=v^3(L)=v^5(0)=v^5(L)=0.
\end{equation}
By inserting \eqref{p3-f1} and \eqref{p3-f5} in \eqref{p3-f2} and \eqref{p3-f6}, system \eqref{p3-f1}-\eqref{p3-f6} implies:
\begin{eqnarray}
-k_1\left(v^1_x+v^3+lv^5 \right)_x -lk_3(v^5_x-lv^1)&=&\rho_1 f^2+ld(x)(-f^5_x+lf^1),\label{p3-f2*}\\
-k_2v^{3}_{xx} +k_1(v^1_x+v^3+lv^5)&=&\rho_2 f^4,\label{p3-f4*}\\
-\left[k_3\left(v^5_x-lv^1\right)+d(x)(-f^5_x+lf^1) \right]_x +lk_1(v^1_x+v^3+lv^5)&=&\rho_1 f^6.\label{p3-f9}
\end{eqnarray}
	Let $(\phi^1 ,\phi^2,\phi^3) \in \left(H^{1}_0 (0,L)\right)^3$. Multiplying   \eqref{p3-f2*}, \eqref{p3-f4*} and \eqref{p3-f9} by $\overline{\phi^1}$, $\overline{\phi^2}$ and $\overline{\phi^3}$ respectively, integrating over $(0,L)$, then using formal integrations by parts, we obtain
	\begin{equation}\label{p3-vf}
	\mathcal{B}((v^1,v^3,v^5),(\phi^1,\phi^2,\phi^3))=\mathcal{L}((\phi^1,\phi^2,\phi^3)), \ \ \forall (\phi^1,\phi^2,\phi^3)\in \left(H^1_0 (0,L)\right)^3,
	\end{equation}
	where
$$
\begin{array}{lll}
&&\displaystyle 	\mathcal{B}((v^1,v^3,v^5),(\phi^1,\phi^2,\phi^3))=\displaystyle k_1\intdx (v^1_x+v^3 +lv^5)\overline{\phi^1_x }dx -lk_3\intdx (v^5_x -lv^1)\overline{\phi^1}dx + k_2\intdx v^3_x \overline{\phi^2_x}dx \vspace{0.25cm}\\
&&\displaystyle   +\,k_1\int_{0}^L (v^1_x +v^3 +lv^5)\overline{\phi^2}dx 
  +k_3\intdx (v^5_x -lv^1)\overline{\phi^3_x}dx +lk_1\intdx (v^1_x+v^3 +lv^5)\overline{\phi^3}dx
\end{array}
$$
and
$$
\begin{array}{lll}
\displaystyle \mathcal{L}((\phi^1,\phi^2,\phi^3))=\rho_1 \intdx f^2 \overline{\phi^1}dx+l\intdx d(x)(-f^5_x+lf^1)\overline{\phi^1}dx+\rho_2 \intdx f^4 \overline{\phi^2}dx\vspace{0.25cm}\\\hspace{3cm}\displaystyle  +\,\intdx d(x)(f^5_x-lf^1) \overline{\phi^3_x}dx+\rho_1 \intdx f^6 \overline{\phi^3}dx.
\end{array}
$$
It is easy to see that  $\mathcal{B}$ is a sesquilinear and  continuous  form on $\left( H^{1}_0 (0,L)\right)^3  \times \left(H^{1}_0 (0,L)\right)^3 $ and $\mathcal{L}$ is a linear and continuous form on $\left( H^{1}_0 (0,L)\right)^3 $. In fact, from Remark \ref{p3-rknorm}, we deduce that there exists a positive constant $c$ such that 
\begin{equation}
\begin{array}{lll}
\mathcal{B}((v^1,v^3,v^5),(v^1,v^3,v^5))=k_1  \|v^1_x+v^3+lv^5\|^2 +k_2 \|v^3_x\|^2 +k_3  \|v^5_x-lv^1\|^2 \vspace{0.25cm}\\
\hspace{4cm}\geq  c \left(\|v^1_x\|^2+\|v^3_x\|^2+\|v^5_x\|^2 \right)\vspace{0.25cm}\\
\hspace{4cm}	=c \left\|(v^1,v^3,v^5)\right\|^2_{\left(H_0^1 (0,L)\right)^3}.
\end{array}
\end{equation}
Thus, $\mathcal{B}$ is a coercive form on $\left(H^1_0(0,L)\right)^3\times \left(H^1_0(0,L)\right)^3$.  Then, it follows by Lax-Milgram theorem that \eqref{p3-vf} admits a unique solution $(v^1,v^3,v^5)\in \left( H^{1}_0 (0,L)\right)^3 $. By taking test-functions 
 $(\phi^1,\phi^2,\phi^3)\in \left(\mathcal{D} (0,L)\right)^3$, we see that
 \eqref{p3-f2*}-\eqref{p3-f9} hold in the distributional sense, from which we deduce that    $(v^1,v^3)\in  \left(H^2 (0,L)\cap H^{1}_0 (0,L)\right)^2$, while   $ \left[k_3v^5_x +d(x)(v^6_x-lv^2) \right]_x\in L^2(0,L)$. Consequently,  $U=(v^1,-f^1,v^3,-f^3,v^5,-f^5)^{\top} \in D(\AA) $ is the unique solution of \eqref{p3-AU=F}. Then, $\mathcal{A}$ is an isomorphism and since $\rho\left(\mathcal{A}\right)$ is open set of $\mathbb{C}$ (see Theorem 6.7 (Chapter III) in \cite{Kato01}),  we easily get $R(\lambda I -\mathcal{A}) = {\mathcal{H}}$ for a sufficiently small $\lambda>0 $. This, together with the dissipativeness of $\mathcal{A}$, imply that   $D\left(\mathcal{A}\right)$ is dense in ${\mathcal{H}}$   and that $\mathcal{A}$ is m-dissipative in ${\mathcal{H}}$ (see Theorems 4.5, 4.6 in  \cite{Pazy01}). The proof is thus complete.
\end{proof}\\\linebreak
According to Lumer-Philips theorem (see \cite{Pazy01}), Proposition \ref{p3-mdissip} implies that the operator $\AA$ generates a $C_{0}$-semigroup of contractions $e^{t\AA}$ in $\HH$ which gives the well-posedness of \eqref{p3-firstevo}. Then, we have the following result:
\begin{Thm}{\rm
For all $U_0 \in \HH$,  system \eqref{p3-firstevo} admits a unique weak solution $$U(t)=e^{t\AA}U_0\in C^0 (\R^+ ,\HH).
	$$ Moreover, if $U_0 \in D(\AA)$, then the system \eqref{p3-firstevo} admits a unique strong solution $$U(t)=e^{t\AA}U_0\in C^0 (\R^+ ,D(\AA))\cap C^1 (\R^+ ,\HH).$$}
\end{Thm}

\section{Strong Stability}\label{p3-sec3}
\noindent In this section, we will prove the strong stability of  system \eqref{p3-sysorig}-\eqref{p3-initialcond}. The main result of this section is the following theorem.
\begin{theoreme}\label{p3-strongthm2}
	{\rm	The $C_0-$semigroup of contraction $\left(e^{t\AA}\right)_{t\geq 0}$ is strongly stable in $\HH$; i.e., for all $U_0\in \HH$, the solution of \eqref{p3-firstevo} satisfies 
		$$
		\lim_{t\rightarrow +\infty}\|e^{t\AA}U_0\|_{\HH}=0.
		$$}
\end{theoreme}\noindent According to Theorem \ref{App-Theorem-A.2}, to prove Theorem \ref{p3-strongthm2}, we need to prove that the operator $\AA$ has no pure imaginary eigenvalues and $\sigma(\AA)\cap i\R $ is countable. The proof of Theorem \ref{p3-strongthm2} has been divided into the following two Lemmas.
\begin{lem}\label{p3-ker}
	{\rm For all $\la \in \R$, $i\la I-\AA$ is injective i.e.
		$$\ker(i\la I-\AA)=\{0\},\ \  \forall \la \in \R.$$

	}
\end{lem}
\begin{proof}
	From Proposition \ref{p3-mdissip}, we have $0\in \rho (\AA)$. We still need to show the result for $\la \in \R^{*}$. For this aim, suppose that there exists a real number $\la\neq0$ and $U=(v^1,v^2,v^3,v^4,v^5,v^6)^{\top}\in D(\AA)$ such that 
	\begin{equation}\label{p2-AU=ilaU}
	\AA U=i\la U.
	\end{equation}Equivalently, we have the following system
	\begin{eqnarray}
	v^2=i\la v^1 \label{p3-f1ker},
	\\\displaystyle k_1(v^{1}_x+v^3+lv^5 )_x+lk_3(v^{5}_x-lv^1 )+ld(x)(v^6_x-lv^2)=i\la \rho_1 v^2\label{p3-f2ker},
	\\\displaystyle v^4=i\la  v^3\label{p3-f3ker},
	\\\displaystyle k_2v^{3}_{xx}  -k_1(v^{1}_x+v^3+lv^5 )=i\la \rho_2v^4\label{p3-f4ker},
	\\\displaystyle v^6=i\la v^5 \label{p3-f5ker},\vspace{0.15cm}
	\\\displaystyle \left[k_3(v^{5}_x-lv^1 )+d(x)(v^6_x-lv^2) \right]_x-lk_1(v^{1}_x+v^3+lv^5)=i\la \rho_1  v^6
	\vspace{0.15cm}\label{p3-f6ker}.
	\end{eqnarray} 
	From  \eqref{p3-Reau}, \eqref{p2-AU=ilaU} and the definition of $d(x)$, we obtain 
	\begin{equation}\label{p2-dissi=0}
	0=\Re \left(i\la U,U\right)_\HH=\Re\left(\AA U,U\right)_{\HH}=-\intdx d(x)\left|v^6_x-lv^2\right|^2 dx =-d_0\int_{\alpha}^{\beta} \left|v^6_x-lv^2\right|^2 dx.
	\end{equation}
	Thus, we have
	\begin{equation}\label{p3-2.43}
v^6_x-lv^2 =0\ \ \text{ in}\ \ (\alpha,\beta ).
	\end{equation}
	Inserting \eqref{p3-f1ker} and \eqref{p3-f5ker} in \eqref{p3-2.43} and using the fact that $\la \neq 0$, we get 
	\begin{equation}\label{p3-2.44}
	v^5_x-lv^1=0 \ \ \text{ in}\ \ (\alpha,\beta ).
	\end{equation}
	Now, inserting \eqref{p3-2.43} and \eqref{p3-2.44} in \eqref{p3-f2ker} and \eqref{p3-f6ker}, then inserting \eqref{p3-f1ker}, \eqref{p3-f3ker} and \eqref{p3-f5ker} in \eqref{p3-f2ker}, \eqref{p3-f4ker} and \eqref{p3-f6ker} respectively, we deduce that 
	\begin{eqnarray}
	\rho_1 	\la^2 v^1+k_1(v^1_x+v^3+lv^5)_x=0 & \text{in} & (\alpha,\beta),\label{p3-3.11}\\
		\rho_2	\la^2 v^3+k_2v^3_{xx}-k_1(v^1_x+v^3+lv^5)=0 & \text{in} & (\alpha,\beta),\label{p3-3.12}\\
			\rho_1	\la^2 v^5-lk_1(v^1_x+v^3+lv^5)=0 & \text{in} & (\alpha,\beta).\label{p3-3.13}
	\end{eqnarray}
	Deriving \eqref{p3-3.13} with respect to $x$, we get
	$$
	\rho_1	\la^2 v^5_x-lk_1(v^1_x+v^3+lv^5)_x=0 \ \ \text{in} \  \ (\alpha,\beta).
	$$
	Inserting  \eqref{p3-3.11} in the above equation, we get 
	\begin{equation}\label{p3-3.14}
	\rho_1	\la^2 (v^5_x+lv^1)=0 \ \ \text{in} \ \ (\alpha,\beta ) \ \ \text{and consequently as } \ \la \neq 0, \ \ \text{we get} \ \ v^5_x +lv^1=0 \ \ \text{in} \ \ (\alpha,\beta ).
	\end{equation}
	Now, adding \eqref{p3-2.44} and \eqref{p3-3.14}, we obtain 
	\begin{equation}\label{p3-3.15}
		v^5_x =0 \  \ \text{in} \ \ (\alpha,\beta) \ \ \text{and consequently } \ \ v^1=0 \ \ \text{in} \ \ (\alpha,\beta).
	\end{equation}
	Inserting \eqref{p3-3.15} in \eqref{p3-3.11}, we get 
	\begin{equation}\label{p3-3.16}
	v^3_x =0 \ \ \text{in} \ \  (\alpha,\beta).
	\end{equation}
	Now, system \eqref{p3-f1ker}-\eqref{p3-f6ker} can be written in $(0,\alpha)\cup(\beta,L)$ as the following:
		\begin{eqnarray}
\displaystyle \rho_1 \la^2v^1+ k_1(v^{1}_x+v^3+lv^5 )_x+lk_3(v^{5}_x-lv^1 )=0 \ \ \text{in} \ \ (0,\alpha)\cup(\beta,L),\label{p3-f1kerp}
	\\\displaystyle \rho_2 \la^2 v^3+k_2v^{3}_{xx}  -k_1(v^{1}_x+v^3+lv^5 )=0\ \ \text{in} \ \ (0,\alpha)\cup(\beta,L),\label{p3-f2kerp}
\vspace{0.15cm}
	\\\displaystyle \rho_1 \la^2 v^5+ k_3(v^{5}_x-lv^1 )_x-lk_1(v^{1}_x+v^3+lv^5)=0\ \ \text{in} \ \ (0,\alpha)\cup(\beta,L).
	\vspace{0.15cm}\label{p3-f3kerp}
	\end{eqnarray} 
	Let $V=(v^1_x,v^1_{xx},v^3_x,v^3_{xx},v^5_x,v^5_{xx})^{\top}$. From \eqref{p3-3.15}, \eqref{p3-3.16} and the regularity of $v^i $, $i\in \{1,3,5\}$, we have $V(\alpha)=0$.
	Now, by deriving system \eqref{p3-f1kerp}-\eqref{p3-f3kerp} with respect to $x$ in $(0,\alpha)$, we deduce that
	 \begin{equation}\label{p3-3.20}
	 V_x =A_\la V \ \ \text{in} \ \ (0,\alpha),
	 \end{equation}
	 where
	 \begin{equation}\label{p3-A}
	 A_\la =  \begin{pmatrix}
 	 0&1&0&0&0&0\\
 	 \frac{l^2k_3-\la^2\rho_1}{k_1}&0&0&-1&0& -l(1+\frac{k_3}{k_1})\\
	 0&0&0&1&0&0\\
	 0&\frac{k_1}{k_2}&\frac{k_1-\rho_2\la^2}{k_2}&0&\frac{lk_1}{k_2}&0\\
 	 0&0&0&0&0&1\\
	 0&l(\frac{k_1}{k_3}+1)&l\frac{k_1}{k_3}&0&\frac{l^2k_1-\rho_1\la^2}{k_3}&0
	 \end{pmatrix}.
	 \end{equation}
	 The solution of the differential equation \eqref{p3-3.20} is given by
	 \begin{equation}\label{p3-solde}
	 V(x)=e^{A_\la (x-\alpha )}V (\alpha ).
	 \end{equation}
	 Thus, from \eqref{p3-solde} and the fact that $V(\alpha)=0$, we get 
	 \begin{equation}\label{p3-2.55}
	 V=0\ \ \text{in}\ \ (0,\alpha).
	 \end{equation}
	 From \eqref{p3-2.55} and the fact that $v^1(0)=v^3(0)=v^5(0)=0$, we get
	 \begin{equation}\label{p3-3.23}
	 v^1=0 \ \ \text{in} \ \ (0,\alpha), \ \ 	 v^3=0 \ \ \text{in} \ \ (0,\alpha)
\ \ \text{and} \ \ 	 v^5=0 \ \ \text{in} \ \ (0,\alpha).
	 \end{equation}
	 From \eqref{p3-3.23}, \eqref{p3-f1ker}, \eqref{p3-f3ker}, \eqref{p3-f5ker} and the fact that $\la\neq 0$, we obtain  
	 \begin{equation}\label{p3-3.24}
	 	U=0 \ \ \text{in} \ \ (0,\alpha).
	 \end{equation}
	 From \eqref{p3-3.24} and the regularity of $v^i$, $i\in \{3,5\}$, we obtain
	 $$
	 v^3(\alpha)=0 \ \ \text{and} \ \ v^5 (\alpha)=0,
	 $$
	 consequently, from \eqref{p3-3.15} and \eqref{p3-3.16}, we get 
	 \begin{equation*}
	 v^1=0 \ \ \text{in} \ \ (\alpha,\beta), \ \ 	 v^3=0 \ \ \text{in} \ \ (\alpha,\beta) \ \ \text{and} \ \ 	 v^5=0 \ \ \text{in} \ \ (\alpha,\beta),
	 \end{equation*}
	 consequently, from  \eqref{p3-f1ker}, \eqref{p3-f3ker}, \eqref{p3-f5ker} and the fact that $\la\neq 0$, we obtain  
	 \begin{equation}\label{p3-3.25}
	  U=0 \ \ \text{in} \ \ (\alpha,\beta).
	 \end{equation}
Now, let $W=(v^1,v^1_{x},v^3,v^3_{x},v^5,v^5_{x})^{\top}$. From \eqref{p3-3.25} and the regularity of $v^i $, $i\in \{1,3,5\}$, we have $W(\beta)=0$ and system \eqref{p3-f1kerp}-\eqref{p3-f3kerp} in $(\beta,L)$ implies:
	$$
	W_x=A_\la W \ \ \text{in} \ \ (\beta,L),
	$$
	where $A_\la$ is defined before  (see \eqref{p3-A}).
	Thus, we have
	\begin{equation*}
	W(x)=e^{A_\la (x-\beta )}W (\beta )=0,
	\end{equation*}
	consequently, from \eqref{p3-f1ker}, \eqref{p3-f3ker} and \eqref{p3-f5ker}, we deduce that 
\begin{equation}\label{p3-3.26}
	U=0 \ \ \text{in} \ \ (\beta,L).
\end{equation}
	Finally, from \eqref{p3-3.24}, \eqref{p3-3.25} and \eqref{p3-3.26}, we obtain
	$$
	U=0 \  \ \text{in} \ \ (0,L).
	$$
	The proof is thus complete.
\end{proof}
\begin{lem}\label{p3-surj}
	{\rm For all $\la \in \R $, we have 
	$$R(i\la I-\AA )=\HH.$$
	
}
\end{lem}
\begin{proof}
		From Proposition \ref{p3-mdissip}, we have $0\in\rho(\AA)$. We still need to show the result for $\la \in \R^{*}$. For this aim, let $F=(f^1,f^2,f^3,f^4,f^5 ,f^6)^{\top}\in \HH$, we want to find $U=(v^1,v^2,v^3,v^4,v^5,v^6)^{\top}\in D(\AA)$ solution of \begin{equation}\label{p3-ilaU-AU=F}
	(	i\la I -\AA)U=F.
	\end{equation}
	Detailing \eqref{p3-ilaU-AU=F}, we obtain
\begin{eqnarray}
i\la  v^1-v^2&=&f^1,\label{p3-f1surji}\\
i\la v^2-\frac{k_1}{\rho_1}\left(v^1_x+v^3+lv^5 \right)_x -\frac{lk_3}{\rho_1}(v^5_x-lv^1)-\frac{ld(x)}{\rho_1}(v^6_x-lv^2)&=& f^2,\label{p3-f2surj}\\
i\la v^3-v^4&=&f^3,\label{p3-f3surj}\\
i\la  v^4-\frac{k_2}{\rho_2}v^{3}_{xx} +\frac{k_1}{\rho_2}(v^1_x+v^3+lv^5)&=& f^4,\label{p3-f4surj}\\
i\la v^5-v^6&=&f^5,\label{p3-f5surj}\\
i\la  v^6-\frac{1}{\rho_1}\left[k_3\left(v^5_x-lv^1\right)+d(x)(v^6_x -lv^2)\right]_x +\frac{lk_1}{\rho_1}(v^1_x+v^3+lv^5)&=& f^6,\label{p3-f6surj}
\end{eqnarray}
with the following boundary conditions
\begin{equation}\label{p3-bcsurj}
	v^1(0)=v^1(L)=v^3(0)=v^3(L)=v^5(0)=v^5(L)=0.
\end{equation}
Inserting $v^2=i\la v^1-f^1$, $v^4=i\la v^3-f^3$ and $v^6=i\la v^5-f^5$ in \eqref{p3-f2surj}, \eqref{p3-f4surj} and \eqref{p3-f6surj} respectively, we obtain
\begin{eqnarray}
-\la^2   v^1-\frac{k_1}{\rho_1}\left(v^1_x+v^3+lv^5 \right)_x -\frac{lk_3}{\rho_1}(v^5_x-lv^1)-\frac{i\la ld(x)}{\rho_1}( v^5_x- lv^1)&=&g^1,\label{p3-f7surj}\\
-\la^2   v^3-\frac{k_2}{\rho_2}v^{3}_{xx} +\frac{k_1}{\rho_2}(v^1_x+v^3+lv^5)&=&g^2,\label{p3-f8surj}\\
-\la^2  v^5-\frac{1}{\rho_1}\left[k_3(v^5_x-lv^1)+i\la d(x)(v^5_x-lv^1) \right]_x +\frac{lk_1}{\rho_1}(v^1_x+v^3+lv^5)&=&g^3,\label{p3-f9surj}
\end{eqnarray}
where
\begin{equation}
\left\{\begin{array}{lll}
\displaystyle	g^1 :=i\la  f^1 +  f^2+\frac{ld(x)}{\rho_1}(-f^5_x+lf^1)\in  H^{-1}(0,L), \ \ g^2:= i\la   f^3+ f^4 \in  H^{-1}(0,L), \vspace{0.25cm}\\ \displaystyle 
	 g^3:= i\la  f^5+ f^6+\rho_1^{-1}\left[d(x)(-f^5_x+lf^1)\right]_x\in H^{-1}(0,L).
	 \end{array}
	 \right.
\end{equation}
For all $\mathsf{U}=(v^1,v^3,v^5)^{\top}\in \mathbb{H}:=\left(H^1_0 (0,L)\right)^3$, we define the  linear operator $\mathbb{L}:\mathbb{H} \longmapsto \mathbb{H}^{\prime}:=\left(H^{-1}(0,L)\right)^3$ by:
\begin{equation}\label{p3-3.40}
	\mathbb{L}\mathsf{U}=\begin{pmatrix}
\displaystyle	-\frac{k_1}{\rho_1 }\left(v^1_x+v^3+lv^5 \right)_x -\frac{lk_3}{\rho_1}(v^5_x-lv^1)-\frac{i\la ld(x)}{\rho_1}( v^5_x- lv^1)\vspace{0.15cm}\\
\displaystyle	-\frac{k_2}{\rho_2}v^{3}_{xx} +\frac{k_1}{\rho_2}(v^1_x+v^3+lv^5)\vspace{0.15cm}\\
\displaystyle	-\frac{1}{\rho_1}\left[k_3(v^5_x-lv^1)+i\la d(x)(v^5_x-lv^1) \right]_x +\frac{lk_1}{\rho_1}(v^1_x+v^3+lv^5)
	\end{pmatrix}.
\end{equation}
Let us prove that the operator $\mathbb{L}$ is an isomorphism. For this aim, take the duality bracket $\langle \cdot,\cdot \rangle_{\mathbb{H}^\prime,\mathbb{H}} $ of \eqref{p3-3.40} with $\Psi:= (\rho_1 \psi^1,\rho_2 \psi^2,\rho_1 \psi^3)^\top \in \mathbb{H}$, we obtain
\begin{equation*}
\begin{array}{lll}
\displaystyle \langle \mathbb{L}\mathsf{U},\Psi \rangle_{\mathbb{H}^\prime,\mathbb{H}} = \left \langle -k_1\left(v^1_x+v^3+lv^5 \right)_x -lk_3(v^5_x-lv^1)-i\la ld(x)( v^5_x- lv^1),\psi^1\right\rangle_{H^{-1}(0,L),H^1_0(0,L)}\vspace{0.25cm}\\
\displaystyle  +\, \left \langle -k_2v^{3}_{xx} +k_1(v^1_x+v^3+lv^5),\psi^2\right\rangle_{H^{-1}(0,L),H^1_0(0,L)}\vspace{0.25cm}\\
\displaystyle  +\, \left \langle -\left[k_3(v^5_x-lv^1)+i\la d(x)(v^5_x-lv^1) \right]_x +lk_1(v^1_x+v^3+lv^5),\psi^3\right\rangle_{H^{-1}(0,L),H^1_0(0,L)}.
\end{array}
\end{equation*}
Consequently, we obtain
$$
\begin{array}{lll}
&&\displaystyle 	\langle \mathbb{L}\mathsf{U},\Psi \rangle_{\mathbb{H}^\prime,\mathbb{H}}=\displaystyle k_1\intdx (v^1_x+v^3 +lv^5)\overline{\psi^1_x }dx -lk_3\intdx (v^5_x -lv^1)\overline{\psi^1}dx-i\la l\intdx d(x)(v^5_x-lv^1)\overline{\psi^1}dx  \vspace{0.25cm}\\
&&\displaystyle   +\, k_2\intdx v^3_x \overline{\psi^2_x}dx+k_1\int_{0}^L (v^1_x +v^3 +lv^5)\overline{\psi^2}dx 
+k_3\intdx (v^5_x -lv^1)\overline{\psi^3_x}dx +i\la \intdx d(x)(v^5_x-lv^1)\overline{\psi^3_x }dx\vspace{0.25cm}\\
&&\displaystyle +\, lk_1\intdx (v^1_x+v^3 +lv^5)\overline{\psi^3}dx,
\end{array}
$$
defines a continuous sesquilinear form which is coercive on $\mathbb{H}$. Indeed, from Remark \ref{p3-rknorm}, we deduce that there exists a positive constant $c^\prime$ such that 

\begin{equation*}
\begin{array}{lll}
\Re\
\displaystyle 	\langle \mathbb{L}\mathsf{U},\mathsf{U} \rangle_{\mathbb{H}^\prime,\mathbb{H}}=k_1  \|v^1_x+v^3+lv^5\|^2 +k_2 \|v^3_x\|^2 +k_3  \|v^5_x-lv^1\|^2 \vspace{0.25cm}\\
\hspace{2.5cm}\geq  c^\prime \left(\|v^1_x\|^2+\|v^3_x\|^2+\|v^5_x\|^2 \right)\vspace{0.25cm}\\
\hspace{2.5cm}	=c^\prime \left\|\left(v^1,v^3,v^5\right)\right\|^2_{\mathbb{H}}\vspace{0.25cm}\\ \hspace{2.5cm}= c^\prime \|\mathsf{U}\|^2_{\mathbb{H}}.
\end{array}
\end{equation*}
Therefore, by using Lax-Milgram theorem, we deduce that  $\mathbb{L}$ is an isomorphism from $\mathbb{H}$ onto $\mathbb{H}^\prime$.
\\\linebreak
Now, let $\mathsf{U}=(v^1,v^3,v^5)^{\top}$ and $\mathsf{G}=(g^1,g^2,g^3)^{\top}$, then system \eqref{p3-f7surj}-\eqref{p3-f9surj} can be transformed into the following form: 
\begin{equation}\label{p3-tr}
(I-\la^2 \mathbb{L}^{-1}) \mathsf{U}=\mathbb{L}^{-1}\mathsf{G}.	
\end{equation}
Since $I$ is compact operator from $\mathbb{H}$ onto $\mathbb{H}^\prime$ and $\mathbb{L}^{-1}$ is an isomorphism from $\mathbb{H}^\prime$ onto $\mathbb{H}$, the operator $I-\la^2 \mathbb{L}^{-1}$ is Fredholm of index zero.
Then, by Fredholm's alternative, \eqref{p3-tr} admits a unique solution $\mathsf{U} \in \mathbb{H}$ if and only if $I-\la^2 \mathbb{L}^{-1}$ is injective. Let $\mathsf{V}=(\mathsf{v^1},\mathsf{v^3},\mathsf{v^5})^{\top} \in \mathbb{H} $ such that 
\begin{equation}
	\mathsf{V}-\la^2 \mathbb{L}^{-1}\mathsf{V}=0 \iff \la^2 \mathsf{V}-\mathbb{L}\mathsf{V}=0.
\end{equation}
Equivalently, we have 
\begin{eqnarray}
-\la^2  \ve^1-\frac{k_1}{\rho_1}\left(\ve^1_x+\ve^3+l\ve^5 \right)_x -\frac{lk_3}{\rho_1}(\ve^5_x-l\ve^1)-\frac{i\la ld(x)}{\rho_1}( \ve^5_x- l\ve^1)&=&0,\label{p3-f7surj*}\\
-\la^2  \ve^3-\frac{k_2}{\rho_2}\ve^{3}_{xx} +\frac{k_1}{\rho_2}(\ve^1_x+\ve^3+l\ve^5)&=&0,\label{p3-f8surj*}\\
-\la^2   \ve^5-\frac{1}{\rho_1}\left[(k_3+i\la d(x))\ve^5_x -l(k_3+i\la) \ve^1 \right]_x +\frac{lk_1}{\rho_1}(\ve^1_x+\ve^3+l\ve^5)&=&0.\label{p3-f9surj*}
\end{eqnarray}
It is easy to see that if $\mathsf{V}=(\ve^1,\ve^2,\ve^3)^{\top}$ is a solution of \eqref{p3-f7surj*}-\eqref{p3-f9surj*}, then the vector $\mathsf{W}$ defined by 
$$
\mathsf{W}=(\ve^1,i\la \ve^1,\ve^3,i\la \ve^3, \ve^5,i\la \ve^5)^{\top} 
$$
belongs to $D(\AA)$ and satisfies
$$
i\la \mathsf{W}-\AA \mathsf{W}=0.
$$
Thus, by using Lemma \ref{p3-ker}, we obtain $\mathsf{W}=0$ and consequently  $I-\la^2 \mathbb{L}^{-1}$ is injective. Thanks to Fredholm's alternative, \eqref{p3-tr} admits a unique solution $\mathsf{U} \in\mathbb{H}$ and 
$$
\begin{array}{lll}
v^1, v^3 \in H^2 (0,L), \ \
\left[k_3v^5_x+ d(x)(i\la v^5_x-f^5_x-l(i\la v^1-f^1)) \right]_x \in L^2 (0,L).
\end{array}
$$
\noindent Finally, by setting $v^2=i\la v^1-f^1$, $v^4=i\la v^3-f^3$ and $v^6=i\la v^5-f^5$, we deduce that $U\in D(\AA)$ is a unique solution of \eqref{p3-ilaU-AU=F}. The proof is thus complete
\end{proof}
\\\linebreak
\textbf{Proof of Theorem \ref{p3-strongthm2}.} From Lemma \ref{p3-ker}, we obtain the the operator $\AA $ has no pure imaginary eigenvalues (i.e. $\sigma_p (\AA)\cap i\R=\emptyset$). Moreover, from Lemma \ref{p3-surj} and with the help of the closed graph theorem of Banach, we deduce that $\sigma(\AA )\cap i\R=\emptyset$. Therefore, according to Theorem \ref{App-Theorem-A.2}, we get that the C$_0 $-semigroup $(e^{t\AA})_{t\geq0}$ is strongly stable. The proof is thus complete. \xqed{$\square$}

\section{Polynomial Stability }\label{p3-secpoly}
\noindent In this section, we will prove the polynomial stability of  system \eqref{p3-sysorig}-\eqref{p3-initialcond} with different rates. The main results of this section are the following theorems. 
\begin{theoreme}\label{p3-pol-eq}{\rm
		 If $$\displaystyle \frac{k_1}{\rho_1}=\frac{k_2}{\rho_2},$$ then there exists $C>0$ such that for every $U_{0}\in D(\AA)$, we have 
		$$
		E(t)\leq \frac{C}{t}\|U_0\|^2_{D(\AA)},\quad t>0.
		$$}
\end{theoreme}
\begin{theoreme}\label{p3-pol-neq}{\rm
		If $$\displaystyle \frac{k_1}{\rho_1}\neq \frac{k_2}{\rho_2},$$ then there exists $C>0$ such that for every $U_{0}\in D(\AA)$, we have 
		$$
		E(t)\leq \frac{C}{\sqrt{t}}\|U_0\|^2_{D(\AA)},\quad t>0.
		$$}
\end{theoreme}
\noindent Since $i\R\subset \rho(\AA)$ (see Section \ref{p3-sec3}), according to  Theorem \ref{bt}, to prove  Theorem \ref{p3-pol-eq} and Theorem \ref{p3-pol-neq}, we still need to prove the following condition  
\begin{equation}\tag{${\rm H}$}\label{p3-H-cond}
\sup_{\la\in \R}\left\|\left(i\la I-\AA\right)^{-1}\right\|_{\mathcal{L}(\HH)}=O\left(\abs{\la}^{\ell}\right), \quad \text{with} \quad \ell=2 \ \ \text{or}\ \ \ell=4.
\end{equation}
We will prove condition \eqref{p3-H-cond} by a contradiction argument. For this purpose,
suppose that \eqref{p3-H-cond} is false, then there exists $\left\{(\la^n,U^n:=(v^{1,n},v^{2,n},v^{3,n},v^{4,n},v^{5,n},v^{6,n})^{\top})\right\}_{n\geq1}\subset \R^{\ast} \times D(\mathcal{A})$ with
\begin{equation}\label{p3-contra-pol2}
|\la^n|\to\infty \quad \text{and}\quad \|U^n\|_{\mathcal{H}}=\|(v^{1,n},v^{2,n},v^{3,n},v^{4,n},v^{5,n},v^{6,n})^{\top}\|_{\HH}=1,
\end{equation}
such that  
\begin{equation}\label{p3-eq0ps}
(\la^n )^{\ell} (i\la^n I-\AA )U^n =F^n:=(f^{1,n},f^{2,n},f^{3,n},f^{4,n},f^{5,n},f^{6,n})^{\top}  \to 0  \quad \text{in}\quad \HH.
\end{equation} 
For simplicity, we drop the index $n$. Equivalently, from \eqref{p3-eq0ps}, we have

\begin{eqnarray}
	i\la v^1 -v^2 &=&\la^{-\ell}f^1,\label{p3-f1ps}\\
	i\la \rho_1 v^2 -k_1 (v^1_x +v^3+lv^5)_x -lk_3 (v^5_x-lv^1)-ld(x)(v^6_x-lv^2)&=&\rho_1 \la^{-\ell}f^2,\label{p3-f2ps}\\
	i\la v^3-v^4&=&\la^{-\ell}f^3,\label{p3-f3ps}\\
	i\la \rho_2 v^4 -k_2 v^3_{xx}+k_1 (v^1_x+v^3+lv^5)&=&\rho_2 \la^{-\ell}f^4,\label{p3-f4ps}\\
	i\la v^5-v^6&=&\la^{-\ell}f^5,\label{p3-f5ps}\\
	i\la \rho_1 v^6 -\left[k_3 (v^5_x-lv^1)+d(x)(v^6_x-lv^2)\right]_x +lk_1 (v^1_x+v^3+lv^5)&=&\rho_1 \la^{-\ell}f^6. \label{p3-f6ps}
\end{eqnarray}
By inserting \eqref{p3-f1ps} in \eqref{p3-f2ps}, \eqref{p3-f3ps} in \eqref{p3-f4ps} and \eqref{p3-f5ps} in \eqref{p3-f6ps}, we deduce that
\begin{eqnarray}
\la^2 \rho_1 v^1 +k_1 (v^1_x +v^3+lv^5)_x +lk_3 (v^5_x-lv^1)+ld(x)(v^6_x-lv^2)&=&-\rho_1 \la^{-\ell}f^2-i\rho_1 \la^{-\ell+1}f^1,\label{p3-4.18*}\\
\la^2 \rho_2 v^3 +k_2 v^3_{xx} -k_1(v^1_x+v^3+lv^5)&=&-\rho_2 \la^{-\ell}f^4-i\rho_2 \la^{-\ell+1}f^3,\label{p3-4.18**}\\
\la^2\rho_1 v^5 +\left[k_3 (v^5_x-lv^1)+d(x)(v^6_x-lv^2)\right]_x -lk_1 (v^1_x+v^3+lv^5)&=&-\rho_1 \la^{-\ell}f^6 -i\rho_1\la^{-\ell+1}  f^5.\label{p3-4.44*}
\end{eqnarray}
Here we will check the condition \eqref{p3-H-cond} by finding a contradiction with \eqref{p3-contra-pol2} by showing $\left\|U\right\|_{\HH}=o(1)$.  For clarity, we divide the proof into several Lemmas. From the above system and the fact that $\ell \in \{2,4\}$, $\|U\|_\HH=1$ and $\|F\|_\HH=o(1)$, we remark that 
\begin{equation}\label{p3-v135}
\left\{\begin{array}{lll}
\displaystyle \|v^1\|=O\left(\left|\la\right|^{-1}\right), \ \ \|v^3\|=O\left(\left|\la\right|^{-1}\right), \ \ \|v^5\|=O\left(\left|\la\right|^{-1}\right), \ \ \|v^1_{xx}\|=O\left(\left|\la\right|\right), \ \ \|v^3_{xx}\|=O\left(\left|\la\right|\right) \ \ \vspace{0.25cm}\\ 
\displaystyle \left\|\left[k_3 (v^5_x-lv^1)+d(x)(v^6_x-lv^2)\right]_x \right\|=O\left(\left|\la\right|\right).
\end{array}\right.
\end{equation}
Also, from Poincar\'e inequality and the fact that $\|F\|_\HH=o(1)$, we remark that 
\begin{equation}\label{p3-poincare}
	\|f^1\|\lesssim \|f^1_x\|=o(1), \ \ 	\|f^3\|\lesssim \|f^3_x\|=o(1) \ \ \text{and} \ \ 	\|f^5\|\lesssim \|f^5_x\|=o(1).
\end{equation}
\begin{lem}\label{p3-1stlemps}
		{\rm If $\left(\frac{k_1}{\rho_1}=\frac{k_2}{\rho_2}\ \ \text{and} \ \ \ell=2\right)$ or $\left(\frac{k_1}{\rho_1}\neq \frac{k_2}{\rho_2} \ \ \text{and} \ \ \ell=4\right)$. Then, the solution $U=(v^1,v^2,v^3,v^4,v^5,v^6)^{\top}\in D(\AA)$ of  \eqref{p3-f1ps}-\eqref{p3-f6ps} satisfies the following estimations 
		\begin{equation}\label{p3-v65xx}
			\int_{\alpha}^{\beta } \left|v^6_x-lv^2\right|^2dx=\frac{o(1)}{\la^{\ell}}, \ \ \int_{\alpha}^{\beta } \left|v^5_x-lv^1\right|^2dx=\frac{o(1)}{\la^{\ell+2}}, \ \ \int_{\alpha}^{\beta } \left|v^6_x\right|^2dx=O(1) \ \  \text{and}\ \ \int_{\alpha}^{\beta } \left|v^5_x\right|^2dx=\frac{O(1)}{\la^2}.
		\end{equation}
	}
\end{lem}
\begin{proof}
		First, taking the inner product of \eqref{p3-eq0ps} with $U$ in $\HH$ and using \eqref{p3-Reau}, we get
	\begin{equation}\label{p3-4.10}
	\displaystyle	\intdx d(x)\left|v^6_x-lv^2\right|^2dx =d_0 \intdnd \left|v^6_x-lv^2\right|^2dx=-\Re \left(\AA U,U\right)_{\HH}=\la^{-\ell}\Re  \left(F,U\right)_{\HH} \leq \la^{-\ell} \|F\|_{\HH}\|U\|_{\HH} .
	\end{equation}Thus, from \eqref{p3-4.10}  and the fact that $\|F\|_{\HH}=o(1)$ and $\|U\|_{\HH}=1$, we obtain the first estimation in \eqref{p3-v65xx}. Deriving \eqref{p3-f5ps} with respect to $x$ and multiply \eqref{p3-f1ps} by $l$, then subtract the resulting equations, we deduce that 
	$$
	i\la (v^5_x-lv^1)-(v^6_x-lv^2)=\la^{-\ell}(f^5_x-lf^1).
	$$
	From the above equation, we obtain
	\begin{equation}\label{p3-4.11}
	\begin{array}{lll}
	\displaystyle	\intdnd \left|v^5_x-lv^1\right|^2dx \leq \frac{2}{\la^{2}}	\intdnd \left|v^6_x-lv^2\right|^2dx+\frac{2}{\la^{2\ell+2}}	\intdnd \left|f^5_x-lf^1\right|^2dx\vspace{0.25cm}\\ \hspace{3cm}\displaystyle \leq \frac{2}{\la^{2}}	\intdnd \left|v^6_x-lv^2\right|^2dx+\frac{4}{\la^{2\ell+2}}\|f^5_x\|^2 +\frac{4l^2}{\la^{2\ell+2}} \|f^1\|^2.
		\end{array}
	\end{equation}
	From \eqref{p3-4.11}, the first estimation in \eqref{p3-v65xx} and the fact that $\ell\in\{2,4\}$, $\|f^1\|=o(1)$ (see \eqref{p3-poincare}), $\|f^5_x\|=o(1)$, we get the second estimation in \eqref{p3-v65xx}. Now, it is easy to see that
	$$
	\intdnd |v^6_x|^2 dx =\intdnd |v^6_x-lv^2+lv^2|^2dx \leq 2 \intdnd |v^6_x-lv^2|^2 dx +2  l^2 \intdnd |v^2|^2dx 
	$$
	From the above estimation, the first estimation in \eqref{p3-v65xx} and the fact that $v^2$ is uniformly bounded in $L^2 (0,L)$, we get the third estimation in \eqref{p3-v65xx}. From \eqref{p3-f5ps}, we deduce that 
	$$
	\intdnd \left|v^5_x\right|^2dx \leq \frac{2}{\la^{2}}	\intdnd \left|v^6_x\right|^2dx+\frac{2}{\la^{2\ell+2}}	\intdnd \left|f^5_x\right|^2dx
	$$
	Finally, from the above estimation, the third estimation in \eqref{p3-v65xx} and the fact that $\|f^5_x\|=o(1)$, we obtain the fourth estimation in \eqref{p3-v65xx}.
	  The proof is thus complete.
\end{proof}\\\linebreak
For all $\displaystyle 0<\varepsilon<\frac{\beta-\alpha}{10}$, we fix the following cut-off functions \\ \linebreak
\begin{itemize}
\item  $\f_j\in C^{2}\left([0,L]\right)$, $j\in \{1,\cdots,5\}$ such that $0\leq \f_j (x)\leq 1$, for all $x\in[0,L]$ and 
\begin{equation*}
\f_j (x)= 	\left \{ \begin{array}{lll}
1 &\text{if} \quad \,\,  x \in [\alpha+j\varepsilon ,\beta -j\varepsilon],&\vspace{0.1cm}\\
0 &\text{if } \quad x \in [0,\alpha+(j-1)\varepsilon]\cup [\beta +(1-j)\varepsilon ,L].&
\end{array}	\right. \qquad\qquad
\end{equation*}\\
\item  $\q_1, \q_2 \in C^{1}\left([0,L]\right)$ such that $0\leq \q_1 (x)\leq 1$, \  $0\leq \q_2 (x)\leq 1$, for all $x\in[0,L]$ and 
\begin{equation*}
\q_1 (x)= 	\left \{ \begin{array}{lll}
1 &\text{if} \quad \,\,  x \in [0,\gamma_1] ,&\vspace{0.1cm}\\
0 &\text{if } \quad x \in [\gamma_2,L ],&
\end{array} \text{and}	\right. 
\ \ 
\q_2 (x)= 	\left \{ \begin{array}{lll}
0 &\text{if} \quad \,\,  x \in [0,\gamma_1] ,&\vspace{0.1cm}\\
1 &\text{if } \quad x \in [\gamma_2,L ],&
\end{array}	\right. \text{with} \ \ 0<\alpha<\gamma_1 <\gamma_2 <\beta <L.
\end{equation*}
\\ 
\end{itemize}

\begin{lem}\label{p3-2ndlemps*}
	{\rm If $\left(\frac{k_1}{\rho_1}=\frac{k_2}{\rho_2}\ \ \text{and} \ \ \ell=2\right)$ or $\left(\frac{k_1}{\rho_1}\neq \frac{k_2}{\rho_2} \ \ \text{and} \ \ \ell=4\right)$. Then, the solution $U=(v^1,v^2,v^3,v^4,v^5,v^6)^{\top}\in D(\AA)$ of  \eqref{p3-f1ps}-\eqref{p3-f6ps} satisfies the following estimations 
		\begin{equation}\label{p3-4.12*}
	\int_{\alpha+\varepsilon}^{\beta-\varepsilon}|v^6|^2 dx =o(1) \ \ \text{and} \ \ \int_{\alpha+\varepsilon}^{\beta-\varepsilon}|\la v^5|^2 dx =o(1).
		\end{equation}	
		
	}
\end{lem}
\begin{proof}
 First, multiplying \eqref{p3-f6ps} by $-i\la^{-1}\f_1\overline{v^6}$ and integrating over $(\alpha,\beta)$, then using the fact that $v^6$ is uniformly bounded in $L^2 (0,L)$ and $\|f^6\|=o(1)$, we obtain 
$$
 \rho_1 \intdnd \f_1 |v^6|^2 dx =-\frac{i}{\la}\intdnd \f_1 \left[k_3 (v^5_x-lv^1)+d(x)(v^6_x-lv^2)\right]_x\overline{v^6}dx +\frac{ilk_1}{\la} \intdnd \f_1 (v^1_x+v^3+lv^5)\overline{v^6}dx +\frac{o(1)}{|\la|^{\ell+1}},
$$
using the fact that $(v^1_x+v^3+lv^5)$, $v^6$ are uniformly bounded in $L^2 (0,L)$, we get  
$$
\frac{ilk_1}{\la} \intdnd \f_1 (v^1_x+v^3+lv^5)\overline{v^6}dx=o(1),
$$
consequently, as $\ell \in \{2,4\}$, we obtain 
\begin{equation}\label{p3-4.13*}
\rho_1 \intdnd \f_1 |v^6|^2 dx =\underbrace{\frac{i}{\la}\intdnd - \f_1 \left[k_3 (v^5_x-lv^1)+d(x)(v^6_x-lv^2)\right]_x\overline{v^6}dx}_{:=\mathtt{I}_1} +o(1). 
\end{equation}
Using integration by parts and the fact that $\f_1(\alpha)=\f_1(\beta)=0$, then using the definition of $d(x)$, we get 
\begin{equation*}
	\mathtt{I}_1 = \frac{i}{\la}\intdnd \f_1 \left[k_3 (v^5_x-lv^1)+d_0(v^6_x-lv^2)\right]\overline{v^6_x}dx+\frac{i}{\la}\intdnd \f_1^{\prime} \left[k_3 (v^5_x-lv^1)+d_0(v^6_x-lv^2)\right]\overline{v^6}dx,
\end{equation*}
using Lemma \ref{p3-1stlemps} and the fact that $v^6$ is uniformly bounded in $L^2 (0,L)$, $\ell \in \{2,4\}$, we get 
\begin{equation}\label{p3-4.15*}
	\mathtt{I}_1 =\frac{o(1)}{|\la|^{\frac{\ell}{2}+1}}.
\end{equation}
Inserting \eqref{p3-4.15*} in \eqref{p3-4.13*} and using the fact that $\ell \in \{2,4\}$, we obtain 
$$
\rho_1 \intdnd \f_1 |v^6|^2 dx=o(1).
$$
From the above estimation and the definition of $\f_1$, we obtain the first estimation in \eqref{p3-4.12*}. Next, from \eqref{p3-f5ps}, we deduce that 
$$
\int_{\alpha+\varepsilon}^{\beta-\varepsilon}|\la v^5|^2 dx \leq 2 \int_{\alpha+\varepsilon}^{\beta-\varepsilon}|v^6|^2dx+2\la^{-2\ell} \int_{\alpha+\varepsilon}^{\beta-\varepsilon}|f^5|^2 dx.
$$
Finally, from the above inequality, the first estimation in \eqref{p3-4.12*} and the fact that $\|f^5\|=o(1)$, $\ell \in \{2,4\}$, we obtain the second estimation in \eqref{p3-4.12*}. The proof is thus complete.
\end{proof}

\begin{lem}\label{p3-2ndlemps}
	{\rm If $\left(\frac{k_1}{\rho_1}=\frac{k_2}{\rho_2}\ \ \text{and} \ \ \ell=2\right)$ or $\left(\frac{k_1}{\rho_1}\neq \frac{k_2}{\rho_2} \ \ \text{and} \ \ \ell=4\right)$. Then, the solution $U=(v^1,v^2,v^3,v^4,v^5,v^6)^{\top}\in D(\AA)$ of  \eqref{p3-f1ps}-\eqref{p3-f6ps} satisfies the following estimations 
	\begin{equation}\label{p3-4.12}
	\int_{\alpha+2\varepsilon}^{\beta-2\varepsilon}\left|v^1_x\right|^2 dx =o(1) \quad \text{and} \quad \int_{\alpha+2\varepsilon}^{\beta-2\varepsilon}\left|\la v^1\right|^2 dx =o(1).
	\end{equation}	
		
	}
\end{lem}
\begin{proof}
First, multiplying \eqref{p3-f6ps} by $\mathsf{f}_2\overline{v^1_x}$, integrating over $(\alpha+\varepsilon,\beta-\varepsilon)$,  using the fact that $v^1_x$ is uniformly bounded in $L^2(0,L)$ and $\|f^6\|=o(1)$, we get 
	\begin{equation*}\label{p3-4.13}
	\begin{array}{lll}
	\displaystyle	\underbrace{i\la \rho_1 \int_{\alpha+\varepsilon}^{\beta-\varepsilon} \f_2 v^6\overline{v^1_x}dx}_{:=\mathtt{I}_2}\  +\underbrace{\int_{\alpha+\varepsilon}^{\beta-\varepsilon} -\f_2 \left[k_3(v^5_x-lv^1) +d(x)(v^6_x-lv^2)\right]_x\overline{v^1_x }dx }_{:=\mathtt{I}_3} \vspace{0.25cm}\\ \displaystyle +\,lk_1 \int_{\alpha+\varepsilon}^{\beta-\varepsilon} \f_2 \left|v^1_x\right|^2 dx +lk_1 \int_{\alpha+\varepsilon}^{\beta-\varepsilon}\f_2 (v^3+lv^5)\overline{v^1_x }dx =o(\la^{-\ell}),
			\end{array}
	\end{equation*}
	using the fact that $v^1_x$ is uniformly bounded in $L^2(0,L)$, $\|v^3\|=O(|\la|^{-1})$, $\|v^5\|=O(|\la|^{-1})$ (see \eqref{p3-v135}), we get 
	$$
	lk_1 \int_{\alpha+\varepsilon}^{\beta-\varepsilon}\f_2 (v^3+lv^5)\overline{v^1_x }dx =o(1),
	$$
	consequently, as $\ell \in \{2,4\}$, we obtain
	\begin{equation}\label{p3-4.16}
lk_1 \int_{\alpha+\varepsilon}^{\beta-\varepsilon} \f_2 \left|v^1_x\right|^2 dx+\mathtt{I}_2+\mathtt{I}_3=o(1).
	\end{equation}
	Now, using integration by parts and the definition of $\f_2$, then using Lemma \ref{p3-2ndlemps*} and the fact that $\|v^1\|=O(|\la|^{-1})$, we get
	\begin{equation}\label{p3-4.17**}
		\mathtt{I}_2 =-i\la \rho_1 \int_{\alpha+\varepsilon}^{\beta-\varepsilon}\f_2 v^6_x \overline{v^1}dx-i\la \rho_1 \int_{\alpha+\varepsilon}^{\beta-\varepsilon}\f_2^{\prime} v^6 \overline{v^1}dx=-i\la \rho_1 \int_{\alpha+\varepsilon}^{\beta-\varepsilon}\f_2 v^6_x \overline{v^1}dx+o(1).
	\end{equation}
	Now, it is easy to see that 
	\begin{equation*}
	\begin{array}{lll}
	\displaystyle	-i\la \rho_1 \int_{\alpha+\varepsilon}^{\beta-\varepsilon}\f_2 v^6_x \overline{v^1}dx=-i\la \rho_1 \int_{\alpha+\varepsilon}^{\beta-\varepsilon}\f_2 (v^6_x-lv^2+lv^2) \overline{v^1}dx \vspace{0.25cm} \\ \hspace{3.5cm}\displaystyle =  -i\la \rho_1 \int_{\alpha+\varepsilon}^{\beta-\varepsilon}\f_2 (v^6_x-lv^2) \overline{v^1}dx-i\la \rho_1 l \int_{\alpha+\varepsilon}^{\beta-\varepsilon}\f_2 v^2 \overline{v^1}dx,
		\end{array}
	\end{equation*}
	using Lemma \ref{p3-1stlemps} and the fact that $\|v^1\|=O(|\la|^{-1})$, we get 
	\begin{equation*}
		-i\la \rho_1 \int_{\alpha+\varepsilon}^{\beta-\varepsilon}\f_2 v^6_x \overline{v^1}dx=-i\la \rho_1 l \int_{\alpha+\varepsilon}^{\beta-\varepsilon}\f_2 v^2 \overline{v^1}dx+o(|\la|^{-\frac{\ell}{2}}).
	\end{equation*}
	Inserting $v^2=i\la v^1-\la^{-\ell}f^1$ in the above equation, we get 
	\begin{equation*}
		-i\la \rho_1 \int_{\alpha+\varepsilon}^{\beta-\varepsilon}\f_2 v^6_x \overline{v^1}dx=l\rho_1 \int_{\alpha+\varepsilon}^{\beta-\varepsilon}\f_2 |\la v^1|^2dx +i\la^{-\ell+1}l\rho_1 \int_{\alpha+\varepsilon}^{\beta-\varepsilon}f^1 \overline{v^1}dx +o(|\la|^{-\frac{\ell}{2}}),
	\end{equation*}
	using the fact that $\|v^1\|=O(|\la|^{-1})$ and $\|f^1\|=o(1)$, $\ell \in \{2,4\}$, we get 
		\begin{equation}\label{p3-4.21*}
	-i\la \rho_1 \int_{\alpha+\varepsilon}^{\beta-\varepsilon}\f_2 v^6_x \overline{v^1}dx=l\rho_1 \int_{\alpha+\varepsilon}^{\beta-\varepsilon}\f_2 |\la v^1|^2dx  +o(|\la|^{-\frac{\ell}{2}}),
	\end{equation}
	Inserting
	\eqref{p3-4.21*} in \eqref{p3-4.17**} and using the fact that $\ell \in \{2,4\}$, we get
	\begin{equation}\label{p3-4.22*}\mathtt{I}_2=l\rho_1 \int_{\alpha+\varepsilon}^{\beta-\varepsilon}\f_2 |\la v^1|^2dx+o(1).
	\end{equation} 
Next, using integration by parts and the definition of $\f_2$, we get 
\begin{equation}\label{p3-4.14}
\mathtt{I}_3=\int_{\alpha+\varepsilon}^{\beta-\varepsilon} \f_2^{\prime} \left[k_3(v^5_x-lv^1) +d_0(v^6_x-lv^2)\right]\overline{v^1_x }dx+\int_{\alpha+\varepsilon}^{\beta-\varepsilon} \f_2 \left[k_3(v^5_x-lv^1) +d_0(v^6_x-lv^2)\right]\overline{v^1_{xx} }dx,
\end{equation}
using Lemma \ref{p3-1stlemps} and the fact that $v^1_x$ is uniformly bounded in $L^2(0,L)$, $\|v^1_{xx}\|=O(|\la|)$ (see \eqref{p3-v135}), $\ell \in \{2,4\}$, we get 
\begin{equation}\label{p3-4.24}
	\mathtt{I}_3=o(|\la|^{-\frac{\ell}{2}+1})
\end{equation}
Inserting \eqref{p3-4.22*} and \eqref{p3-4.24} in \eqref{p3-4.16} and using the fact that $\ell \in \{2,4\}$, we get 
\begin{equation}
lk_1 \int_{\alpha+\varepsilon}^{\beta-\varepsilon} \f_2 \left|v^1_x\right|^2 dx+l\rho_1 \int_{\alpha+\varepsilon}^{\beta-\varepsilon}\f_2 |\la v^1|^2dx=o(1).
\end{equation}
Finally, from the above estimation and the definition of $\f_2$, we obtain \eqref{p3-4.12}. The proof is thus complete.
\end{proof}
\begin{lem}\label{p3-3rdlem}
	{\rm If $\frac{k_1}{\rho_1}=\frac{k_2}{\rho_2}$ and $\ell=2$. Then, the solution $U=(v^1,v^2,v^3,v^4,v^5,v^6)^{\top}\in D(\AA)$ of  \eqref{p3-f1ps}-\eqref{p3-f6ps} satisfies the following estimations 
 \begin{equation}\label{p3-4.17*}
\int_{\alpha+3\varepsilon}^{\beta-3\varepsilon}\left|v^3_x\right|^2dx =o(1) \quad \text{and} \quad \int_{\alpha+4\varepsilon}^{\beta-4\varepsilon}\left|\la v^3\right|^2dx =o(1).
 \end{equation}
}
\end{lem}
\begin{proof}
First, take $\ell =2$ in \eqref{p3-4.18*} and multiply it by $\rho_1^{-1}\f_3\overline{v^3_x}$, integrating over $(\alpha+2
\varepsilon,\beta-2\varepsilon)$, using the definition of $d(x)$ and fact that $v^3_x$ is uniformly bounded in $L^2(0,L)$, $\|f^1\|=o(1)$, $\|f^2\|=o(1)$, we obtain
\begin{equation}\label{p3-4.21}
\begin{array}{lll}
\displaystyle \frac{k_1}{\rho_1}\int_{\alpha+2\varepsilon}^{\beta-2\varepsilon} \f_3 \left|v^3_x\right|^2 dx =-\la^2 \int_{\alpha+2\varepsilon}^{\beta-2\varepsilon}  \f_3 v^1 \overline{v^3_x}dx-\frac{k_1}{\rho_1}\int_{\alpha+2\varepsilon}^{\beta-2\varepsilon}  \f_3 v^1_{xx}\overline{v^3_x}dx -\frac{lk_1}{\rho_1}\int_{\alpha+2\varepsilon}^{\beta-2\varepsilon}  \f_3 v^5_x\overline{v^3_x}dx   \vspace{0.25cm}\\ \hspace{3.5cm}\displaystyle -\,\frac{lk_3}{\rho_1}\int_{\alpha+2\varepsilon}^{\beta-2\varepsilon}  \f_3 (v^5_x-lv^1)\overline{v^3_x}dx -\frac{ld_0}{\rho_1}\int_{\alpha+2\varepsilon}^{\beta-2\varepsilon}  \f_3 (v^6_x-lv^2)\overline{v^3_x}dx +o(|\la|^{-1}).
\end{array}
\end{equation}
Using Lemma \ref{p3-1stlemps} with $\ell=2$, the definition of $\f_3$ and the fact that $v^3_x$ is uniformly bounded in $L^2(0,L)$, we get 
\begin{equation}\label{p3-4.22}
\left\{\begin{array}{lll}
\displaystyle -	\frac{lk_1}{\rho_1} \int_{\alpha+2\varepsilon}^{\beta-2\varepsilon} \f_3 v^5_x \overline{v^3_x }dx=o(1), \ \    -\frac{lk_3}{\rho_1}\int_{\alpha+2\varepsilon}^{\beta-2\varepsilon}  \f_3 (v^5_x-lv^1)\overline{v^3_x}dx=o(\la^{-2}) \ \ \text{and} \vspace{0.25cm}\\ \displaystyle -\frac{ld_0}{\rho_1}\int_{\alpha+2\varepsilon}^{\beta-2\varepsilon}  \f_3 (v^6_x-lv^2)\overline{v^3_x}dx=o(|\la|^{-1}).
\end{array}
\right.
\end{equation}
Inserting \eqref{p3-4.22} in \eqref{p3-4.21}, we get 
\begin{equation}\label{p3-4.23}
	\frac{k_1}{\rho_1}\intdx \f_3 \left|v^3_x\right|^2 dx =-\la^2 \intdx \f_3 v^1 \overline{v^3_x}dx-\frac{k_1}{\rho_1}\intdx \f_3 v^1_{xx}\overline{v^3_x}dx+o(1).
\end{equation}
Now, taking $\ell=2$ in \eqref{p3-4.18**}, we deduce that  
\begin{equation}\label{p3-4.20}
\la^2 \rho_2 \overline{v^3} +k_2 \overline{v^3_{xx}}-k_1 (\overline{v^1_x} +\overline{v^3}+l\overline{v^5})=-\rho_2 \la^{-2}\overline{f^4}+i\rho_2\la^{-1}\overline{f^3}.
\end{equation}
Multiplying \eqref{p3-4.20} by $\rho_2^{-1}\f_3v^1_x$, integrating over $(\alpha+2\varepsilon,\beta-2\varepsilon)$, we obtain
\begin{equation}
	\la^2 \int_{\alpha+2\varepsilon}^{\beta-2\varepsilon} \f_3 v^1_x \overline{v^3 }dx +\frac{k_2}{\rho_2}\int_{\alpha+2\varepsilon}^{\beta-2\varepsilon} \f_3 v^1_{x}\overline{v^3_{xx}}dx -\frac{k_1}{\rho_2}\int_{\alpha+2\varepsilon}^{\beta-2\varepsilon} \f_3 v^1_x (\overline{v^1_x} +\overline{v^3} +l\overline{v^5})dx=o(|\la|^{-1}) 
\end{equation}
Using integration by parts to the first two terms in the above equation, we get 
\begin{equation}\label{p3-4.25}
\begin{array}{lll}
\displaystyle-\la^2 \int_{\alpha+2\varepsilon}^{\beta-2\varepsilon} \f_3 v^1 \overline{v^3_x }dx -\frac{k_2}{\rho_2}\int_{\alpha+2\varepsilon}^{\beta-2\varepsilon} \f_3 v^1_{xx}\overline{v^3_{x}}dx= \la^2 \int_{\alpha+2\varepsilon}^{\beta-2\varepsilon} \f_3^{\prime} v^1 \overline{v^3}dx+\frac{k_2}{\rho_2}\int_{\alpha+2\varepsilon}^{\beta-2\varepsilon} \f_3^{\prime} v^1_{x}\overline{v^3_{x}}dx\vspace{0.25cm}\\\displaystyle +\, \frac{k_1}{\rho_2}\int_{\alpha+2\varepsilon}^{\beta-2\varepsilon} \f_3 v^1_x (\overline{v^1_x} +\overline{v^3} +l\overline{v^5})dx+o(|\la|^{-1}).
\end{array}
\end{equation}
Using Lemma \ref{p3-2ndlemps} and the fact that $v^3_x $, $(v^1_x+v^3+lv^5)$ are uniformly bounded in $L^2 (0,L)$ and $\|v^3\|=O(|\la|^{-1})$, we get 
\begin{equation}\label{p3-4.26}
\displaystyle\la^2 \int_{\alpha+2\varepsilon}^{\beta-2\varepsilon}\f_3^{\prime} v^1 \overline{v^3}dx =o(1), \ \ \frac{k_2}{\rho_2}\int_{\alpha+2\varepsilon}^{\beta-2\varepsilon} \f_3^{\prime} v^1_{x}\overline{v^3_{x}}dx=o(1), \ \ \frac{k_1}{\rho_2}\int_{\alpha+2\varepsilon}^{\beta-2\varepsilon} \f_3 v^1_x (\overline{v^1_x} +\overline{v^3} +l\overline{v^5})dx=o(1).
\end{equation}
Inserting \eqref{p3-4.26} in \eqref{p3-4.25}, then using the fact that $\displaystyle \frac{k_2}{\rho_2}=\frac{k_1}{\rho_1}$, we get 
\begin{equation*}
 -\la^2 \int_{\alpha+2\varepsilon}^{\beta-2\varepsilon} \f_3 v^1 \overline{v^3_x }dx -\frac{k_1}{\rho_1}\int_{\alpha+2\varepsilon}^{\beta-2\varepsilon} \f_3 v^1_{xx}\overline{v^3_{x}}dx=o(1).
\end{equation*}
Inserting the above estimation in \eqref{p3-4.23}, then using the definition of $\f_3$, we obtain the first estimation in \eqref{p3-4.17*}. Next, multiplying \eqref{p3-4.20} by $\f_4 v^3$, integrating over $(\alpha+3\varepsilon,\beta-3\varepsilon)$, using integration by parts and the definition of $\f_4$ and the fact that $\|v^3\|=O(|\la|^{-1})$, $\|f^3\|=o(1)$ and $\|f^4\|=o(1)$, we get
\begin{equation*}
	\rho_2 \int_{\alpha+3\varepsilon}^{\beta-3\varepsilon} \f_4 \left|\la v^3\right|^2 dx =k_2 \int_{\alpha+3\varepsilon}^{\beta-3\varepsilon} \f_4 |v^3_x|^2dx+k_2 \int_{\alpha+3\varepsilon}^{\beta-3\varepsilon} \f_4^{\prime} \overline{v^3_{x}}v^3 dx +k_1 \int_{\alpha+3\varepsilon}^{\beta-3\varepsilon} \f_4 (\overline{v^1_x} +\overline{v^3}+l\overline{v^5})v^3 dx +o(\la^{-2}).
\end{equation*}
From the above estimation, the first estimation in \eqref{p3-4.17*} and the fact that $(v^1_x+v^3+lv^5)$ is uniformly bounded in $L^2 (0,L)$ and $\|v^3\|=O(|\la|^{-1})$, we obtain  
$$
	\rho_2 \int_{\alpha+3\varepsilon}^{\beta-3\varepsilon} \f_4 \left|\la v^3\right|^2 dx=o(1).
$$
Finally, from the above estimation and the definition of $\f_4$, we obtain the second estimation desired. The proof is thus complete.
\end{proof} 
\begin{lem}\label{p3-4thlemma}
 	{\rm If  $\frac{k_1}{\rho_1}\neq \frac{k_2}{\rho_2}$ and $\ell=4$. Then, the solution $U=(v^1,v^2,v^3,v^4,v^5,v^6)^{\top}\in D(\AA)$ of  \eqref{p3-f1ps}-\eqref{p3-f6ps} satisfies the following estimation
	\begin{equation}\label{p3-4.35*}
	\int_{\alpha+3\varepsilon}^{\beta-3\varepsilon}|\la v^1|^2 dx =o(\la^{-2}).
	\end{equation}
}
\end{lem}
\begin{proof} 
	For clarity, we divide the proof into five steps:\\
	\textbf{Step 1:} In this step, we will prove that:
		\begin{equation}\label{p3-4.37}
	\begin{array}{lll}
	\displaystyle	l\rho_1 \int_{\alpha+2\varepsilon}^{\beta-2\varepsilon} \g \left|\la v^1 \right|^2 dx-lk_1 \int_{\alpha+2\varepsilon}^{\beta-2\varepsilon} \g \left|v^1_x\right|^2 dx -\Re\left\{lk_1 \int_{\alpha+2\varepsilon}^{\beta-2\varepsilon} \g v^3 \overline{v^1_x}dx \right\}\vspace{0.25cm}\\\displaystyle -\Re \left\{l^2 k_1 \int_{\alpha+2\varepsilon}^{\beta-2\varepsilon} \g v^5 \overline{v^1_x}dx  \right\}=o(\la^{-2}).
	\end{array}
	\end{equation}
For this aim, take $\ell=4$ in \eqref{p3-4.18*} and multiply it by $l\g \overline{v^1} $, integrating over $(\alpha+2\varepsilon,\beta-2\varepsilon)$, using the fact that $\|v^1\|=O(|\la|^{-1})$, $\|f^1\|=o(1)$ and $\|f^2\|=o(1)$, then taking the real part, we get 
	\begin{equation}\label{p3-4.31}
	\begin{array}{lll}
	\displaystyle 	l\rho_1  \int_{\alpha+2\varepsilon}^{\beta-2\varepsilon} \g \left|\la v^1 \right|^2 dx +\underbrace{\Re \left\{lk_1 \int_{\alpha+2\varepsilon}^{\beta-2\varepsilon} \g (v^1_x +v^3 +lv^5)_x \overline{v^1 }dx  \right\}}_{:= \mathtt{I}_4}\vspace{0.25cm}\\ \displaystyle +\, \Re \left\{l^2k_3 \int_{\alpha+2\varepsilon}^{\beta-2\varepsilon} \g (v^5_x -lv^1 )\overline{v^1}dx  \right\}+\Re\left\{l^2d_0 \int_{\alpha+2\varepsilon}^{\beta-2\varepsilon}\f_3(v^6_x-lv^2)\overline{v^1}dx  \right\}=o(\la^{-4}).
	\end{array}
	\end{equation}
	Using integration by parts and the definition of $\g$, we obtain
	\begin{equation}\label{p3-4.32}
	\begin{array}{lll}
\displaystyle 	\mathtt{I}_4=-\Re \left\{lk_1 \int_{\alpha+2\varepsilon}^{\beta-2\varepsilon} \g^{\prime} (v^1_x +v^3 +lv^5) \overline{v^1 }dx  \right\}-\Re \left\{lk_1 \int_{\alpha+2\varepsilon}^{\beta-2\varepsilon} \g (v^1_x +v^3 +lv^5) \overline{v^1_x }dx  \right\}\vspace{0.25cm}\\ 
\hspace{0.5cm}\displaystyle =-\frac{lk_1}{2}\int_{\alpha+2\varepsilon}^{\beta-2\varepsilon} \g^{\prime}\left(\left|v^1\right|^2\right)_x dx -\Re \left\{lk_1 \int_{\alpha+2\varepsilon}^{\beta-2\varepsilon} \g^{\prime}v^3 \overline{v^1}dx  \right\}-\Re\left\{l^2 k_1 \int_{\alpha+2\varepsilon}^{\beta-2\varepsilon} \g^{\prime}v^5 \overline{v^1}dx  \right\}\vspace{0.25cm}\\ 
\hspace{0.5cm}\displaystyle -\,lk_1 \int_{\alpha+2\varepsilon}^{\beta-2\varepsilon} \g \left|v^1_x\right|^2 dx -\Re\left\{lk_1 \int_{\alpha+2\varepsilon}^{\beta-2\varepsilon} \g v^3 \overline{v^1_x}dx \right\}-\Re \left\{l^2 k_1 \int_{\alpha+2\varepsilon}^{\beta-2\varepsilon} \g v^5 \overline{v^1_x}dx  \right\}.
	\end{array}
	\end{equation}
	Using integration by parts  and the fact that $\g^{\prime}(\alpha+2\varepsilon)=\g^{\prime}(\beta-2\varepsilon)=0$, then using Lemma \ref{p3-2ndlemps}, we obtain
	\begin{equation}\label{p3-4.33}
	-\frac{lk_1}{2}\int_{\alpha+2\varepsilon}^{\beta-2\varepsilon} \g^{\prime}\left(\left|v^1\right|^2\right)_x dx =	\frac{lk_1}{2}\int_{\alpha+2\varepsilon}^{\beta-2\varepsilon} \g^{\prime \prime}\left|v^1\right|^2 dx  =o(\la^{-2}).
	\end{equation}
	Using the definition of $\g$, Lemmas \ref{p3-1stlemps}, \ref{p3-2ndlemps} with $\ell =4$ and the fact that $\|v^3\|=O(|\la|^{-1})$, $\|v^5\|=O(|\la|^{-1})$, we obtain
	\begin{equation}\label{p3-4.34}
\displaystyle -\Re \left\{lk_1 \int_{\alpha+2\varepsilon}^{\beta-2\varepsilon} \g^{\prime}v^3 \overline{v^1}dx  \right\}=o(\la^{-2}), \ \ -\Re\left\{l^2 k_1 \int_{\alpha+2\varepsilon}^{\beta-2\varepsilon} \g^{\prime}v^5 \overline{v^1}dx  \right\}=o(\la^{-2}).
	\end{equation}
	Inserting \eqref{p3-4.33} and \eqref{p3-4.34} in \eqref{p3-4.32}, we obtain 
	\begin{equation}\label{p3-4.35}
		\mathtt{I}_4=-lk_1 \int_{\alpha+2\varepsilon}^{\beta-2\varepsilon} \g \left|v^1_x\right|^2 dx -\Re\left\{lk_1 \int_{\alpha+2\varepsilon}^{\beta-2\varepsilon} \g v^3 \overline{v^1_x}dx \right\}-\Re \left\{l^2 k_1 \int_{\alpha+2\varepsilon}^{\beta-2\varepsilon} \g v^5 \overline{v^1_x}dx  \right\}.+o(\la^{-2}).
	\end{equation}
	Moreover, from  Lemmas \ref{p3-1stlemps}, \ref{p3-2ndlemps} and the fact that $\ell=4$, we obtain 
	\begin{equation}\label{p3-4.36}
\Re \left\{l^2k_3 \int_{\alpha+2\varepsilon}^{\beta-2\varepsilon}\g (v^5_x -lv^1 )\overline{v^1}dx  \right\}=o(\la^{-4}) \ \ \text{and} \ \ \Re\left\{l^2d_0 \int_{\alpha+2\varepsilon}^{\beta-2\varepsilon}\f_3(v^6_x-lv^2)\overline{v^1}dx  \right\}=o(|\la|^{-3}).
	\end{equation}
	Inserting \eqref{p3-4.35} and \eqref{p3-4.36} in \eqref{p3-4.31}, we obtain \eqref{p3-4.37}.\\ \linebreak
	\textbf{Step 2:} In this step,  we will prove that:
		\begin{equation}\label{p3-4.43**}
	2	l\rho_1 \int_{\alpha+2\varepsilon}^{\beta-2\varepsilon} \g \left|\la v^1 \right|^2 dx=\Re\left\{i\la \rho_1 \int_{\alpha+2\varepsilon}^{\beta-2\varepsilon}\f_3^{\prime} v^6\overline{v^1}dx \right\}-\Re\left\{d_0 \int_{\alpha+2\varepsilon}^{\beta-2\varepsilon}\g (v^6_x-lv^2) \overline{v^1_{xx}}dx\right\}+o(\la^{-2}).
	\end{equation}
For this aim, multiplying \eqref{p3-f6ps} by $\g \overline{v^1_x }$, integrating over $(\alpha+2\varepsilon,\beta-2\varepsilon)$, using the fact that $v^1_x$ is uniformly bounded in $L^2(0,L)$ and  $\|f^6\|=o(1)$,  then taking the real part, we get 
	\begin{equation}\label{p3-4.38}
	\begin{array}{lll}
	\displaystyle \underbrace{\Re\left\{	i\la \rho_1 \int_{\alpha+2\varepsilon}^{\beta-2\varepsilon} \g v^6\overline{v^1_x}dx\right\}}_{:=\mathtt{I}_5} \  +\underbrace{\Re\left\{-\int_{\alpha+2\varepsilon}^{\beta-2\varepsilon} \g \left[k_3(v^5_x-lv^1) +d(x)(v^6_x-lv^2)\right]_x \overline{v^1_x }dx\right\}}_{:=\mathtt{I}_6}\vspace{0.25cm}\\ \displaystyle  +\, lk_1 \int_{\alpha+2\varepsilon}^{\beta-2\varepsilon} \g|v^1_x|^2dx +\Re\left\{ lk_1 \int_{\alpha+2\varepsilon}^{\beta-2\varepsilon} \g v^3\overline{v^1_x }dx\right\}+\Re\left\{ l^2k_1 \int_{\alpha+2\varepsilon}^{\beta-2\varepsilon} \g v^5\overline{v^1_x }dx\right\} =o(\la^{-4}).
	\end{array}
	\end{equation}
	Adding \eqref{p3-4.37} and \eqref{p3-4.38}, we obtain 
	\begin{equation}\label{p3-4.45*}
		l\rho_1 \int_{\alpha+2\varepsilon}^{\beta-2\varepsilon} \g \left|\la v^1 \right|^2 dx+\mathtt{I}_5+\mathtt{I}_6=o(\la^{-2}).
	\end{equation}
		Using integration by parts and the fact that $\g(\alpha+2\varepsilon)=\g(\beta-2\varepsilon)=0$, we obtain 
		\begin{equation}\label{p3-4.46***}
		\mathtt{I}_5=-\Re\left\{i\la \rho_1 \int_{\alpha+2\varepsilon}^{\beta-2\varepsilon}\f_3 v^6_x\overline{v^1}dx \right\}-\Re\left\{i\la \rho_1 \int_{\alpha+2\varepsilon}^{\beta-2\varepsilon}\f_3^{\prime} v^6\overline{v^1}dx \right\}.
		\end{equation}
			Now, it is easy to see that 
		\begin{equation*}
		\begin{array}{lll}
		\displaystyle	\Re\left\{-i\la \rho_1 \int_{\alpha+2\varepsilon}^{\beta-2\varepsilon}\f_2 v^6_x \overline{v^1}dx\right\}=\Re\left\{-i\la \rho_1 \int_{\alpha+2\varepsilon}^{\beta-2\varepsilon}\f_3(v^6_x-lv^2+lv^2) \overline{v^1}dx \right\}\vspace{0.25cm} \\ \hspace{3.5cm}\displaystyle =  \Re\left\{-i\la \rho_1 \int_{\alpha+2\varepsilon}^{\beta-2\varepsilon}\f_3 (v^6_x-lv^2) \overline{v^1}dx\right\}-\Re\left\{i\la \rho_1 l \int_{\alpha+2\varepsilon}^{\beta-2\varepsilon}\f_3v^2 \overline{v^1}dx\right\},
		\end{array}
		\end{equation*}
		using Lemma \ref{p3-1stlemps} and the fact that $\|v^1\|=O(|\la|^{-1})$, we get 
		\begin{equation*}
	\Re\left\{	-i\la \rho_1 \int_{\alpha+2\varepsilon}^{\beta-2\varepsilon}\f_3 v^6_x \overline{v^1}dx\right\}=\Re\left\{-i\la \rho_1 l \int_{\alpha+2\varepsilon}^{\beta-2\varepsilon}\f_3 v^2 \overline{v^1}dx\right\}+o(\la^{-2}).
		\end{equation*}
		Inserting $v^2=i\la v^1-\la^{-4}f^1$ in the above estimation, then 	using the fact that $\|v^1\|=O(|\la|^{-1})$ and $\|f^1\|=o(1)$, we get  
		\begin{equation*}
	\Re\left\{	-i\la \rho_1 \int_{\alpha+2\varepsilon}^{\beta-2\varepsilon}\f_3 v^6_x \overline{v^1}dx\right\}=l\rho_1 \int_{\alpha+2\varepsilon}^{\beta-2\varepsilon}\f_3 |\la v^1|^2dx  +o(\la^{-2}),
		\end{equation*}
		Inserting the above estimation in \eqref{p3-4.46***}, we obtain
		\begin{equation}\label{p3-4.50}
			\mathtt{I}_5=l\rho_1 \int_{\alpha+2\varepsilon}^{\beta-2\varepsilon}\f_3 |\la v^1|^2dx -\Re\left\{i\la \rho_1 \int_{\alpha+2\varepsilon}^{\beta-2\varepsilon}\f_3^{\prime} v^6\overline{v^1}dx \right\}+o(\la^{-2}).
			\end{equation}
			Now, Using integration by parts and the fact that $\g(\alpha+2\varepsilon)=\g(\beta-2\varepsilon)=0$, then using the definition of $d(x)$, we obtain 
	\begin{equation*}
		\begin{array}{lll}
         \mathtt{I}_6=  \displaystyle\Re\left\{\int_{\alpha+2\varepsilon}^{\beta-2\varepsilon} \g^{\prime} \left[k_3(v^5_x-lv^1) +d(x)(v^6_x-lv^2)\right] \overline{v^1_x }dx\right\}+\Re\left\{\int_{\alpha+2\varepsilon}^{\beta-2\varepsilon} \g \left[k_3(v^5_x-lv^1) +d(x)(v^6_x-lv^2)\right] \overline{v^1_{xx} }dx\right\}\vspace{0.25cm} 
          \\ 
          \hspace{0.5cm}\displaystyle =\, \Re\left\{k_3 \int_{\alpha+2\varepsilon}^{\beta-2\varepsilon} \g^{\prime}(v^5_x-lv^1) \overline{v^1_x }dx  \right\}   -\Re\left\{d_0 \int_{\alpha+2\varepsilon}^{\beta-2\varepsilon}\g^{\prime}(v^6_x-lv^2) \overline{v^1_x}dx\right\}\vspace{0.25cm}\\
          \hspace{1cm}\displaystyle
     +\, \Re\left\{k_3 \int_{\alpha+2\varepsilon}^{\beta-2\varepsilon} \g (v^5_x-lv^1) \overline{v^1_{xx}}dx \right\}+\Re\left\{d_0 \int_{\alpha+2\varepsilon}^{\beta-2\varepsilon}\g (v^6_x-lv^2) \overline{v^1_{xx}}dx\right\},
		\end{array}
	\end{equation*}
consequently, by using Lemma \ref{p3-1stlemps} with $\ell=4$ and the fact that $v^1_x$ is uniformly bounded in $L^2(0,L)$, $\|v^1_{xx}\|=O(|\la|)$, we get
	\begin{equation}\label{p3-4.51}
		\mathtt{I}_6	= \Re\left\{d_0 \int_{\alpha+2\varepsilon}^{\beta-2\varepsilon}\g (v^6_x-lv^2) \overline{v^1_{xx}}dx\right\}+o(\la^{-2}).
	\end{equation}
Thus,	by inserting \eqref{p3-4.50} and \eqref{p3-4.51} in \eqref{p3-4.45*}, we obtain \eqref{p3-4.43**}.\\ \linebreak
\textbf{Step 3:} In this step,  we will prove that:
\begin{equation}\label{p3-4.52}
\Re\left\{i\la \rho_1 \int_{\alpha+2\varepsilon}^{\beta-2\varepsilon}\f_3^{\prime} v^6\overline{v^1}dx \right\}=o(\la^{-2}).
\end{equation}
	For this aim, take $\ell=4$ in \eqref{p3-f6ps} and multiply it by $\f_3^\prime$$\overline{v^1}$, integrating over $(\alpha+2\varepsilon,\beta-2\varepsilon)$, using the fact that $\|v^1\|=O(|\la|^{-1})$,  $\|f^6\|=o(1)$,  then taking the real part, we get  
	\begin{equation}\label{p3-4.53}
		\begin{array}{lll}
		\displaystyle \Re \left\{i\la \rho_1 \int_{\alpha+2\varepsilon}^{\beta-2\varepsilon}\f_3^\prime v^6\overline{v^1}dx \right\}+\underbrace{\Re\left\{-\int_{\alpha+2\varepsilon}^{\beta-2\varepsilon} \g^{\prime} \left[k_3(v^5_x-lv^1) +d(x)(v^6_x-lv^2)\right]_x \overline{v^1 }dx\right\}}_{:=\mathtt{I}_7}\vspace{0.25cm}\\ \displaystyle +\frac{lk_1}{2}\int_{\alpha+2\varepsilon}^{\beta-2\varepsilon}\f_3^\prime \left(|v^1|^2\right)_x dx +\Re\left\{ lk_1 \int_{\alpha+2\varepsilon}^{\beta-2\varepsilon}\f_3^\prime (v^3+lv^5)\overline{v^1}dx \right\}=o(|\la|^{-5}),
		\end{array}
	\end{equation}
	using \eqref{p3-4.33}, Lemma \ref{p3-2ndlemps} and the fact that $\|v^3\|=O(|\la|^{-1})$, $\|v^5\|=O(|\la|^{-1})$, we obtain 
	$$
	\frac{lk_1}{2}\int_{\alpha+2\varepsilon}^{\beta-2\varepsilon}\f_3^\prime \left(\left|v^1\right|^2\right)_x dx=o(\la^{-2}) \ \ \text{and} \ \ \Re\left\{ lk_1 \int_{\alpha+2\varepsilon}^{\beta-2\varepsilon}\f_3^\prime (v^3+lv^5)\overline{v^1}dx \right\}=o(\la^{-2}).
	$$
	Consequently, \eqref{p3-4.53} implies
	\begin{equation}\label{p3-4.54}
	 \Re \left\{i\la \rho_1 \int_{\alpha+2\varepsilon}^{\beta-2\varepsilon}\f_3^\prime v^6\overline{v^1}dx \right\}+\mathtt{I}_7=o(\la^{-2}).
	\end{equation}
	Using integration by parts  and the fact that $\g^{\prime}(\alpha+2\varepsilon)=\g^{\prime}(\beta-2\varepsilon)=0$, then using Lemma \ref{p3-1stlemps} and the fact that $v^1_x$ is uniformly bounded in $L^2(0,L)$, $\|v^1\|=O(|\la|^{-1})$, we obtain
	 \begin{equation*}
	 \begin{array}{lll}
\displaystyle	\mathtt{I}_7= \Re\left\{\int_{\alpha+2\varepsilon}^{\beta-2\varepsilon} \g^{\prime\prime} \left[k_3(v^5_x-lv^1) +d(x)(v^6_x-lv^2)\right] \overline{v^1 }dx\right\}+\Re\left\{\int_{\alpha+2\varepsilon}^{\beta-2\varepsilon} \g^{\prime} \left[k_3(v^5_x-lv^1) +d(x)(v^6_x-lv^2)\right] \overline{v^1_x }dx\right\}\vspace{0.25cm}\\ \qquad \displaystyle =\, o(\la^{-2}).
	\end{array}
	 \end{equation*}
	 Therefore, from the above estimation and \eqref{p3-4.54}, we obtain \eqref{p3-4.52}.\\ \linebreak
	 \textbf{Step 4:} In this step, we will prove that:
	 \begin{equation}\label{p3-4.55}
	 \Re\left\{d_0 \int_{\alpha+2\varepsilon}^{\beta-2\varepsilon}\g (v^6_x-lv^2) \overline{v^1_{xx}}dx\right\}=-\Re \left\{ \frac{d_0\rho_1}{k_1}\la^2 \int_{\alpha+2\varepsilon}^{\beta-2\varepsilon}\g (v^6_x-lv^2) \overline{v^1}dx\right\}+o(\la^{-2}).
	 \end{equation}
For this aim, take $\ell=4$ in \eqref{p3-4.18*} and multiply it by $\dfrac{d_0}{k_1} \g (\overline{v^6_x}-l\overline{v^2})$, integrating over $(\alpha+2\varepsilon,\beta-2\varepsilon)$ and taking the real part, then using Lemmas \ref{p3-1stlemps} and the fact that $\|f^1\|=o(1)$, $\|f^2\|=o(1)$, we get 
	\begin{equation*}
\left\{	\begin{array}{lll}
	\displaystyle 	\Re \left\{ \frac{d_0\rho_1}{k_1}\la^2 \int_{\alpha+2\varepsilon}^{\beta-2\varepsilon}\g v^1 (\overline{v^6_x}-l\overline{v^2})dx\right\}+\Re \left\{d_0 \int_{\alpha+2\varepsilon}^{\beta-2\varepsilon}\g v^1_{xx}(\overline{v^6_x}-l\overline{v^2})dx \right\}+\Re\left\{d_0 \int_{\alpha+2\varepsilon}^{\beta-2\varepsilon}\g v^3_x (\overline{v^6_x}-l\overline{v^2})dx \right\}\vspace{0.25cm}\\
	\displaystyle  +\, \Re\left\{d_0 l \int_{\alpha+2\varepsilon}^{\beta-2\varepsilon}\g v^5_x (\overline{v^6_x}-l\overline{v^2})dx\right\}+\Re \left\{\frac{d_0 lk_3}{k_1}\int_{\alpha+2\varepsilon}^{\beta-2\varepsilon}\g (v^5_x-lv^1)(\overline{v^6_x}-l\overline{v^2})dx \right\}+\frac{ld_0^2}{k_1}\int_{\alpha+2\varepsilon}^{\beta-2\varepsilon}\f_3 |v^6_x-lv^2|^2dx\vspace{0.25cm}\\ \displaystyle \, =o(|\la|^{-5}),
		\end{array}
		\right.
	\end{equation*}
consequently, by using Lemma \ref{p3-1stlemps} and the fact that $v^3_x$ is uniformly bounded in $L^2(0,L)$, we get
	\begin{equation*}\label{p3-4.47}
		\Re \left\{d_0 \int_{\alpha+2\varepsilon}^{\beta-2\varepsilon}\g v^1_{xx}(\overline{v^6_x}-l\overline{v^2})dx \right\}=-\Re \left\{ \frac{d_0\rho_1}{k_1}\la^2 \int_{\alpha+2\varepsilon}^{\beta-2\varepsilon}\g v^1 (\overline{v^6_x}-l\overline{v^2})dx\right\}+o(\la^{-2}).
	\end{equation*}
	Thus, from the above estimation, we obtain \eqref{p3-4.55}.\\ \linebreak
	\textbf{Step 5:} In this step, we conclude the proof of \eqref{p3-4.35*}. For this aim, 
	inserting \eqref{p3-4.52} and \eqref{p3-4.55} in \eqref{p3-4.43**}, then using Young's inequality, Lemma \ref{p3-1stlemps} and the fact that $\ell =4$, we get 
	\begin{equation*}
	\begin{array}{lll}
\displaystyle	2l\rho_1  \int_{\alpha+2\varepsilon}^{\beta-2\varepsilon} \g \left|\la v^1 \right|^2 dx=	\Re \left\{ \frac{d_0\rho_1}{k_1}\la^2 \int_{\alpha+2\varepsilon}^{\beta-2\varepsilon}\g (v^6_x-lv^2) \overline{v^1}dx\right\}+o(\la^{-2})\vspace{0.25cm}
\\\hspace{5cm} \displaystyle 
\leq  \, \frac{d_0\rho_1}{k_1}\la^2 \int_{\alpha+2\varepsilon}^{\beta-2\varepsilon}\g|v^6_x-lv^2||v^1|dx +o(\la^{-2})\vspace{0.25cm}
\\\hspace{5cm} \displaystyle = \int_{\alpha+2\varepsilon}^{\beta-2\varepsilon}\left(\frac{d_0 \sqrt{\rho_1}}{k_1 \sqrt{l}}|\la|\sqrt{\g}|v^6_x-lv^2|\right)\left(\sqrt{l\rho_1}|\la|\sqrt{\g}|v^1|\right)dx +o(\la^{-2})\vspace{0.25cm}
\\ \hspace{5cm}\displaystyle \leq \,   \underbrace{\frac{\rho_1 d_0^2}{2k_1^2 l}\la^2 \int_{\alpha+2\varepsilon}^{\beta-2\varepsilon}\g|v^6_x-lv^2|^2 dx}_{=o(\la^{-2})}+\frac{l\rho_1 }{2}\int_{\alpha+2\varepsilon}^{\beta-2\varepsilon}\g|\la v^1|^2 dx +o(\la^{-2}),
	\end{array}
	\end{equation*}
consequently, we obtain
$$
\frac{3l\rho_1 }{2} \int_{\alpha+2\varepsilon}^{\beta-2\varepsilon} \g \left|\la v^1 \right|^2 dx=o(\la^{-2}).
$$
Finally, from the above estimation and the definition of $\g$, we obtain \eqref{p3-4.35*}. The proof is thus complete.
\end{proof}
\begin{lem}\label{p3-5thlemma}
	{\rm If  $\frac{k_1}{\rho_1}\neq \frac{k_2}{\rho_2}$ and $\ell=4$. Then, the solution $U=(v^1,v^2,v^3,v^4,v^5,v^6)^{\top}\in D(\AA)$ of system \eqref{p3-f1ps}-\eqref{p3-f6ps} satisfies the following estimations 
		\begin{equation}\label{p3-4.40}
		\int_{\alpha+4\varepsilon}^{\beta-4\varepsilon}\left| v^3_x\right|^2 dx=o(1) \quad \text{and}\quad   \int_{\alpha+5\varepsilon}^{\beta-5\varepsilon}\left|\la v^3 \right|^2 dx = o(1).
		\end{equation}
	}
\end{lem}
\begin{proof}
	First, take $\ell=4$ in \eqref{p3-4.18*} and multiply it by $k_1^{-1}\f_4 \overline{v^3_x}$, integrating over $(\alpha+3\varepsilon,\beta-3\varepsilon)$, using the definition of $d(x)$ and the fact that $v^3_x$ is uniformly bounded in $L^2 (0,L)$, $\|f^1\|=o(1)$, $\|f^2\|=o(1)$, then taking the real part, we obtain  
\begin{equation*}
\begin{array}{lll}
\displaystyle  \Re\left\{	\frac{\la^2\rho_1}{k_1} \int_{\alpha+3\varepsilon}^{\beta-3\varepsilon}\f_4 v^1 \overline{v^3_x} dx+\int_{\alpha+3\varepsilon}^{\beta-3\varepsilon}\f_4 v^1_{xx}\overline{v^3_x}dx+\int_{\alpha+3\varepsilon}^{\beta-3\varepsilon}\f_4 |v^3_x|^2 dx +l\int_{\alpha+3\varepsilon}^{\beta-3\varepsilon}\f_4 v^5_x \overline{v^3_x}dx\right.\vspace{0.25cm}\\ \hspace{1cm} \displaystyle \left. +\, \frac{lk_3}{k_1}\int_{\alpha+3\varepsilon}^{\beta-3\varepsilon}\f_4 (v^5_x-lv^1)\overline{v^3_x}dx +\frac{ld_0}{k_1}\int_{\alpha+3\varepsilon}^{\beta-3\varepsilon}\f_4(v^6_x-lv^2)\overline{v^3_x}dx \right\} =o(|\la|^{-3}),
\end{array}
\end{equation*}
consequently, from  Lemmas \ref{p3-1stlemps}, \ref{p3-4thlemma} with $\ell=4$ and the fact that $v^3_x$ is uniformly bounded in $L^2(0,L)$, we obtain 
\begin{equation}\label{p3-4.60}
\Re\left\{	\int_{\alpha+3\varepsilon}^{\beta-3\varepsilon}\f_4 v^1_{xx}\overline{v^3_x}dx\right\}+\int_{\alpha+3\varepsilon}^{\beta-3\varepsilon}\f_4 |v^3_x|^2 dx=o(1).
\end{equation}
Now, take $\ell=4$ in \eqref{p3-4.18**} and multiply it by $k_2^{-1}\f_4 \overline{v^1_x}$, integrating over $(\alpha+3\varepsilon,\beta-3\varepsilon)$ and integrating by parts, using  the fact that $v^1_x$ is uniformly bounded in $L^2 (0,L)$ and $\|f^3\|=o(1)$, $\|f^4\|=o(1)$, then taking the real part, we obtain  
\begin{equation*}
\begin{array}{lll}
\displaystyle \Re \left\{ 	-\frac{\la^2\rho_2}{k_2}\int_{\alpha+3\varepsilon}^{\beta-3\varepsilon}\f_4v^3_x\overline{v^1}dx-\frac{\la^2\rho_2}{k_2}\int_{\alpha+3\varepsilon}^{\beta-3\varepsilon}\f_4^\prime v^3\overline{v^1}dx -\int_{\alpha+3\varepsilon}^{\beta-3\varepsilon}\f_4v^3_x \overline{v^1_{xx}}dx - \int_{\alpha+3\varepsilon}^{\beta-3\varepsilon}\f_4^\prime v^3_x \overline{v^1_x}dx\right.\vspace{0.25cm} \\ \hspace{1cm}\left. \displaystyle -\, \frac{k_1}{k_2} \int_{\alpha+3\varepsilon}^{\beta-3\varepsilon}\f_4(v^1_x+v^3+lv^5)\overline{v^1_x}dx\right\}=o(|\la|^{-3}),
\end{array}
\end{equation*}
consequently, from Lemmas \ref{p3-2ndlemps}, \ref{p3-4thlemma} with $\ell=4$ and the fact that $v^3_x$, $(v^1_x+v^3+lv^5)$ are uniformly bounded in $L^2 (0,L)$ and $\|v^3\|=O(|\la|^{-1})$, we obtain 
\begin{equation}\label{p3-4.61}
	\Re\left\{ -\int_{\alpha+3\varepsilon}^{\beta-3\varepsilon}\f_4v^3_x \overline{v^1_{xx}}dx\right\}=o(1).
\end{equation}
Adding \eqref{p3-4.60} and \eqref{p3-4.61}, then using the definition of $\f_4$, we obtain the first estimation in \eqref{p3-4.40}. Next, take $\ell=4$ in \eqref{p3-4.18**} and multiply it by $\f_5\overline{v^3}$, integrating over $(\alpha+4\varepsilon,\beta-4\varepsilon)$ and integrating by parts, then using the fact that $\|v^3\|=O(|\la|^{-1})$, $\|f^3\|=o(1)$, $\|f^4\|=o(1)$, we obtain 
\begin{equation*}
	\rho_2 \int_{\alpha+4\varepsilon}^{\beta-4\varepsilon}\f_5 |\la v^3|^2dx = k_2 \int_{\alpha+4\varepsilon}^{\beta-4\varepsilon}\f_5 |v^3_x|^2 dx +k_2 \int_{\alpha+4\varepsilon}^{\beta-4\varepsilon}\f_5^\prime v^3_x \overline{v^3}dx +k_1 \int_{\alpha+4\varepsilon}^{\beta-4\varepsilon}\f_5(v^1_x+v^3+lv^5)\overline{v^3}dx+o(\la^{-4}).
\end{equation*}
From the above estimation, the first estimation in \eqref{p3-4.40} and the fact that  $(v^1_x+v^3+lv^5)$ is uniformly bounded in $L^2 (0,L)$ and $\|v^3\|=O(|\la|^{-1})$, we obtain
$$
\rho_2 \int_{\alpha+4\varepsilon}^{\beta-4\varepsilon}\f_5|\la v^3|^2 dx=o(1).
$$
Finally, from the above estimation and the definition of $\f_5$, we obtain the second estimation in \eqref{p3-4.40}. The proof is thus complete.
\end{proof}

\begin{lem}\label{p3-6thlemma}
	{\rm  Let $\h \in C^1([0,L])$ such that $\h (0)=\h(L)=0$. If $\left(\frac{k_1}{\rho_1}=\frac{k_2}{\rho_2}\ \ \text{and} \ \ \ell=2\right)$ or $\left(\frac{k_1}{\rho_1}\neq \frac{k_2}{\rho_2} \ \ \text{and} \ \ \ell=4\right)$, 
		 then the solution $U=(v^1,v^2,v^3,v^4,v^5,v^6)^{\top}\in D(\AA)$ of system \eqref{p3-f1ps}-\eqref{p3-f6ps} satisfies the following estimation
		 	\begin{equation*}\label{p3-4.41}
		 \intdx \h^{\prime} \left(\rho_1\left|\la v^1\right|^2+k_1\left|v^1_x \right|^2+\rho_2\left| \la v^3 \right|^2 +k_2 \left|v^3_x\right|^2+\rho_1 \left|\la v^5\right|^2+k_3^{-1}\left|k_3 v^5_x +d(x)(v^6_x-lv^2)\right|^2 \right) dx=o(1).
		 \end{equation*}
	}
\end{lem}
\begin{proof} 
	First, multiplying \eqref{p3-4.18*} by $2\h \overline{v^1_x}$, integrating over $(0,L)$, taking the real part, then using Lemma \ref{p3-1stlemps}, the fact that $v^1_x$ is uniformly bounded in $L^2(0,L)$, $\|v^1\|=O(|\la|^{-1})$, $\|f^1\|=o(1)$ and $\|f^2\|=o(1)$, we obtain 
	\begin{equation}\label{p3-4.42*}
	\begin{array}{lll}
\displaystyle 	\intdx \h \left(\rho_1\left|\la v^1\right|^2+k_1\left|v^1_x \right|^2 \right)_x dx +\Re \left\{2k_1 \intdx \h v^3_x \overline{v^1_x}dx  \right\}+\Re\left\{2l(k_1+k_3)\intdx \h v^5_x \overline{v^1_x}dx \right\}\vspace{0.25cm}\\ \displaystyle
\underbrace{	-\Re\left\{2l^2k_3 \intdx \h v^1\overline{v^1_x}dx  \right\}}_{=o(1)}+\underbrace{\Re\left\{2ld_0 \int_{\alpha}^{\beta}\h (v^6_x-lv^2)\overline{v^1_x}dx \right\}}_{=o\left(|\la|^{-\frac{\ell}{2}}\right)}=o(|\la|^{-\ell+1}).
	\end{array}
	\end{equation}
	Now,  multiplying \eqref{p3-4.18**} by $2\h \overline{v^3_x}$, integrating over $(0,L)$, taking the real part, then using the fact that $v^3_x$ is uniformly bounded in $L^2(0,L)$, $\|v^3\|=O(|\la|^{-1})$, $\|v^5\|=O(|\la|^{-1})$, $\|f^3\|=o(1)$ and $\|f^4\|=o(1)$, we obtain 
	 \begin{equation}\label{p3-4.43*}
	 	\intdx \h \left(\left|\rho_2 \la v^3 \right|^2 +k_2 \left|v^3_x\right|^2\right)_x dx -\Re\left\{2k_1 \intdx \h v^1_x \overline{v^3_x}dx  \right\}\underbrace{-\Re \left\{2k_1\intdx\h (v^3+lv^5)\overline{v^3_x }dx  \right\}}_{=o(1)}=o(|\la|^{1-\ell}).
	 \end{equation}
Let $\mathsf{S}:=k_3 v^5_x +d(x)(v^6_x-lv^2)$, from Lemma \ref{p3-1stlemps}, the definition of $d(x)$ and the fact that $v^5_x$ is uniformly bounded in $L^2 (0,L)$, we get $\s$ is uniformly bounded in $L^2 (0,L)$.	Now, multiplying \eqref{p3-4.44*} by $2k_3^{-1}\h \overline{\mathsf{S}}$, integrating over $(0,L)$, taking the real part, then using the fact that $\|v^3\|=O(|\la|^{-1})$, $\|v^5\|=O(|\la|^{-1})$, $\|f^5\|=o(1)$ and $\|f^6\|=o(1)$, we obtain
	\begin{equation}\label{p3-4.46}
\begin{array}{lll}
\displaystyle \Re\left\{\frac{2\la^2 \rho_1}{k_3} \intdx \h v^5 \overline{\mathsf{S}}dx  \right\}+k_3^{-1}\intdx \h \left(\left|\s\right|^2\right)_x dx -\Re \left\{\frac{2l(k_1+k_3)}{k_3}\intdx \h v^1_x \overline{\s}dx \right\}\vspace{0.25cm}\\ 
\displaystyle \underbrace{-\Re \left\{\frac{2lk_1}{k_3} \intdx \h (v^3+lv^5)\overline{\s}dx  \right\}}_{=o(1)}=\underbrace{\Re\left\{2k_3^{-1}\intdx \h \left(-\rho_1 \la^{-\ell}f^6 -i\la^{1-\ell} \rho_1 f^5\right)\overline{\s}dx \right\}}_{=o\left(\left|\la\right|^{-\ell+1}\right)}.
\end{array}
	\end{equation}
Moreover, from the definition of $\s$ and $d(x)$, Lemma \ref{p3-1stlemps} and the fact that $v^1_x$ is uniformly bounded in $L^2 (0,L)$, $\|v^5\|=O(|\la|^{-1})$, we obtain  
	$$
\left\{\begin{array}{lll}
\displaystyle	\Re\left\{\frac{2\la^2 \rho_1}{k_3} \intdx \h v^5 \overline{\mathsf{S}}dx  \right\}=\la^2 \rho_1 \intdx \h \left(\left|v^5\right|^2\right)_x dx +\underbrace{\Re \left\{ \frac{2\la^2 \rho_1 d_0 }{k_3}\intdnd \h v^5 (\overline{v^6_x}-l\overline{v^2})dx  \right\}}_{=o\left(\left|\la\right|^{-\frac{\ell}{2}+1}\right)},\vspace{0.25cm}\\
\displaystyle 
	-\Re \left\{\frac{2l(k_1+k_3)}{k_3}\intdx \h v^1_x \overline{\s}dx \right\}=-\Re \left\{2l(k_1+k_3)\intdx \h v^1_x \overline{v^5_x }dx \right\}\underbrace{-\Re \left\{\frac{2l(k_1+k_3)d_0}{k_3}\intdnd \h v^1_x (\overline{v^6_x }-l\overline{v^2})dx \right\}}_{=o\left(\left|\la\right|^{-\frac{\ell}{2}}\right)}.
	\end{array}
	\right.
	$$
	Inserting the above estimations in \eqref{p3-4.46} and using the fact that $\ell \in \{2,4\}$, we obtain 
	\begin{equation}\label{p3-4.47*}
	\intdx \h \left(\rho_1 \left|\la v^5\right|^2+k_3^{-1}\left|\s\right|^2 \right)_x dx -\Re \left\{2l(k_1+k_3)\intdx \h v^1_x \overline{v^5_x }dx \right\}=o(1).
	\end{equation}
	Adding \eqref{p3-4.42*}, \eqref{p3-4.43*}, \eqref{p3-4.47*} and using the fact that $\ell \in \{2,4\}$, then using integration by parts, we obtain  \eqref{p3-4.41}. The proof is thus complete.
	\end{proof}
	\begin{lem}
		{\rm
	 The solution $U=(v^1,v^2,v^3,v^4,v^5,v^6)^{\top}\in D(\AA)$ of system \eqref{p3-f1ps}-\eqref{p3-f6ps} satisfies the following estimations
		\begin{equation}\label{p3-4.63}
		\mathsf{J}(\alpha+4\varepsilon,\beta-4\varepsilon)=o(1) \qquad \text{if} \qquad  \frac{k_1}{\rho_1}=\frac{k_2}{\rho_2} \ \ \text{and} \ \ \ell=2,
		\end{equation}
			\begin{equation}\label{p3-4.64}
		\mathsf{J}(\alpha+5\varepsilon,\beta-5\varepsilon)=o(1) \qquad \text{if} \qquad  \frac{k_1}{\rho_1}\neq\frac{k_2}{\rho_2} \ \ \text{and} \ \ \ell=4,
		\end{equation}
		where 
		\begin{equation*}
	\begin{array}{lll}
	\displaystyle \mathsf{J}(\gamma_1,\gamma_2):=	\int_{0}^{\gamma_1} \left(\rho_1\left|\la v^1\right|^2+k_1\left|v^1_x \right|^2+\rho_2\left| \la v^3 \right|^2 +k_2 \left|v^3_x\right|^2 +\rho_1 |\la v^5|^2 \right) dx +k_3\int_{0}^{\alpha} |v^5_x|^2 dx  \vspace{0.25cm}\\
\hspace{2cm}	\displaystyle
	+	\int_{\gamma_2}^{L} \left(\rho_1\left|\la v^1\right|^2+k_1\left|v^1_x \right|^2+\rho_2\left| \la v^3 \right|^2 +k_2 \left|v^3_x\right|^2+\rho_1 \left|\la v^5\right|^2 \right) dx +k_3\int_{\beta}^{L} |v^5_x|^2 dx,
	\end{array}
	\end{equation*}
	for all $0<\alpha<\gamma_1<\gamma_2<\beta<L$.
}
	\end{lem}
\begin{proof}
First, take $\h =x\q_1+(x-L)\q_2$ in \eqref{p3-4.41}, then using the definition of $d(x)$ and the fact that $0<\alpha<\gamma_1<\gamma_2<\beta<L$, we obtain
		\begin{equation*}
	\begin{array}{lll}
\displaystyle	\int_{0}^{\gamma_1} \left(\rho_1\left|\la v^1\right|^2+k_1\left|v^1_x \right|^2+\rho_2\left| \la v^3 \right|^2 +k_2 \left|v^3_x\right|^2 +\rho_1 |\la v^5|^2 \right) dx +k_3\int_{0}^{\alpha} |v^5_x|^2 dx  \vspace{0.25cm}\\
\displaystyle
+	\int_{\gamma_2}^{L} \left(\rho_1\left|\la v^1\right|^2+k_1\left|v^1_x \right|^2+\rho_2\left| \la v^3 \right|^2 +k_2 \left|v^3_x\right|^2+\rho_1 \left|\la v^5\right|^2 \right) dx +k_3\int_{\beta}^{L} |v^5_x|^2 dx\vspace{0.25cm}\\ 
\displaystyle = -\int_{\gamma_1}^{\gamma_2}\left(\q_1+x\q_1^{\prime}\right)\left(\rho_1\left|\la v^1\right|^2+k_1\left|v^1_x \right|^2+\rho_2\left| \la v^3 \right|^2 +k_2 \left|v^3_x\right|^2+\rho_1 \left|\la v^5\right|^2+k_3^{-1}\left|k_3v^5_x+d_0 (v^6_x-lv^2)\right|^2 \right) dx\vspace{0.25cm}\\
\hspace{0.5cm}\displaystyle  -\int_{\gamma_1}^{\gamma_2}\left(\q_2+(x-L)\q_2^{\prime}\right)\left(\rho_1\left|\la v^1\right|^2+k_1\left|v^1_x \right|^2+\rho_2\left| \la v^3 \right|^2 +k_2 \left|v^3_x\right|^2+\rho_1 \left|\la v^5\right|^2+k_3^{-1}\left|k_3v^5_x+d_0 (v^6_x-lv^2)\right|^2 \right) dx\vspace{0.25cm}\\
\hspace{0.5cm}\displaystyle +k_3^{-1}\int_{\alpha}^{\gamma_2} \q_1 |k_3v^5_x+d_0(v^6_x-lv^2)|^2dx
 +k_3^{-1}\int_{\gamma_1}^{\beta} \q_2 |k_3v^5_x+d_0(v^6_x-lv^2)|^2dx.
	\end{array}
	\end{equation*}
	Now, take $\gamma_1=\alpha+4\varepsilon$ and $\gamma_2=\beta-4\varepsilon$ in the above equation, then using Lemmas \ref{p3-1stlemps}-\ref{p3-3rdlem} in case of $\frac{k_1}{\rho_1}=\frac{k_2}{\rho_2}$ and $\ell=2$, we obtain \eqref{p3-4.63}.
	Finally, take $\gamma_1=\alpha+5\varepsilon$ and $\gamma_2=\beta-5\varepsilon$ in the above equation, then using Lemmas \ref{p3-1stlemps}-\ref{p3-2ndlemps}, \ref{p3-5thlemma} in case of $\frac{k_1}{\rho_1}\neq\frac{k_2}{\rho_2}$ and $\ell=4$, we obtain \eqref{p3-4.64}.  The proof is thus complete.
\end{proof}
\\\linebreak
\textbf{Proof of Theorem \ref{p3-pol-eq}}. 
First, from  Lemmas \ref{p3-1stlemps}-\ref{p3-3rdlem} and the fact that $\ell=2$, we obtain 
\begin{equation}\label{p3-4.46*}
\left\{\begin{array}{l}
\displaystyle{\int_\alpha^{\beta}|v^5_x|^2dx=O(\la^{-2})=o(1),\ \int_{\alpha+\varepsilon}^{\beta-\varepsilon}\abs{v^6}^2dx=o(1),}\ \int_{\alpha+2\varepsilon}^{\beta-2\varepsilon}|v^1_x|^2 dx =o(1)\vspace{0.25cm}\\
\displaystyle{\int_{\alpha+2\varepsilon}^{\beta-2\varepsilon}\abs{\la v^1}^2dx=o(1), \ \int_{\alpha+3\varepsilon}^{\beta-3\varepsilon}\abs{v^3_x}^2dx=o(1) \quad \text{and}\quad \int_{\alpha+4\varepsilon}^{\beta-4\varepsilon}\abs{\la v^3}^2dx=o(1)}.
\end{array}
\right.
\end{equation}
Now, from \eqref{p3-4.63}, \eqref{p3-4.46*} and the fact that $\displaystyle 0<\varepsilon<\frac{\beta-\alpha}{10}$, we deduce that $\|U\|_\HH=o(1) \ \ \text{in} \ \ (0,L)$, which contradicts \eqref{p3-H-cond}. This implies that 
$$
\sup_{\la\in \R}\|(i\la I-\AA)^{-1}\|_{\mathcal{\HH}}=O\left(\la^2\right).
$$ 
Finally, according to Theorem \ref{bt}, we obtain the desired result.
The proof is thus complete.\xqed{$\square$}
\\\linebreak
\textbf{Proof of Theorem \ref{p3-pol-neq}}. 
First, from  Lemmas \ref{p3-1stlemps}, \ref{p3-2ndlemps*}, \ref{p3-2ndlemps}, \ref{p3-5thlemma} and the fact that $\ell=4$, we obtain 
\begin{equation}\label{p3-4.46**}
\left\{\begin{array}{l}
\displaystyle{\int_\alpha^{\beta}|v^5_x|^2dx=O(\la^{-2})=o(1),\ \int_{\alpha+\varepsilon}^{\beta-\varepsilon}\abs{v^6}^2dx=o(1),}\ \int_{\alpha+2\varepsilon}^{\beta-2\varepsilon}| v^1_x|^2 dx =o(\la^{-2})\vspace{0.25cm}\\
\displaystyle{\int_{\alpha+2\varepsilon}^{\beta-2\varepsilon}\abs{ \la v^1}^2dx=o(1), \ \int_{\alpha+4\varepsilon}^{\beta-4\varepsilon}\abs{v^3_x}^2dx=o(1) \quad \text{and}\quad \int_{\alpha+5\varepsilon}^{\beta-5\varepsilon}\abs{\la v^3}^2dx=o(1)}.
\end{array}
\right.
\end{equation}
Now, from \eqref{p3-4.64}, \eqref{p3-4.46**} and the fact that $\displaystyle 0<\varepsilon<\frac{\beta-\alpha}{10}$, we deduce that $\|U\|_\HH=o(1) \ \ \text{in} \ \ (0,L)$, which contradicts \eqref{p3-H-cond}. This implies that 
$$
\sup_{\la\in \R}\|(i\la I-\AA)^{-1}\|_{\mathcal{\HH}}=O\left(\la^4\right).
$$ 
Finally, according to Theorem \ref{bt}, we obtain the desired result.
The proof is thus complete.\xqed{$\square$}
	
\section{Conclusion}
We have studied  the stabilization of a Bresse system with one  discontinuous local internal viscoelastic damping of Kelvin-Voigt type acting on the axial force under fully Dirichlet boundary conditions. We proved the strong stability of the system by using Arendt-Batty criteria. We proved that the energy of our system  decays polynomially with the rates:
\begin{equation*}
\left\{	\begin{array}{lll}
\displaystyle	t^{-1} \quad \text{if} \quad \frac{k_1}{\rho_1}=\frac{k_2}{\rho_2},\vspace{0.15cm}\\
\displaystyle 		t^{-\frac{1}{2}} \quad \text{if} \quad \frac{k_1}{\rho_1}\neq\frac{k_2}{\rho_2}.
\end{array}
\right.
\end{equation*}

\appendix
\section{Some notions and stability theorems}\label{p2-appendix}
\noindent In order to make this paper more self-contained, we recall in this short appendix some notions and stability results used in this work. 
\begin{defi}\label{App-Definition-A.1}{\rm
		Assume that $A$ is the generator of $C_0-$semigroup of contractions $\left(e^{tA}\right)_{t\geq0}$ on a Hilbert space $H$. The $C_0-$semigroup $\left(e^{tA}\right)_{t\geq0}$ is said to be 
		\begin{enumerate}
			\item[$(1)$] Strongly stable if 
			$$
			\lim_{t\to +\infty} \|e^{tA}x_0\|_H=0,\quad \forall\, x_0\in H.
			$$
			\item[$(2)$] Exponentially (or uniformly) stable if there exists two positive constants $M$ and $\varepsilon$ such that 
			$$
			\|e^{tA}x_0\|_{H}\leq Me^{-\varepsilon t}\|x_0\|_{H},\quad \forall\, t>0,\ \forall\, x_0\in H.
			$$
			\item[$(3)$] Polynomially stable if there exists two positive constants $C$ and $\alpha$ such that 
			$$
			\|e^{tA}x_0\|_{H}\leq Ct^{-\alpha}\|A x_0\|_{H},\quad \forall\, t>0,\ \forall\, x_0\in D(A).
			$$
			\xqed{$\square$}
	\end{enumerate}}
\end{defi}
\noindent To show  the strong stability of the $C_0$-semigroup $\left(e^{tA}\right)_{t\geq0}$ we rely on the following result due to Arendt-Batty \cite{Arendt01}. 
\begin{Theorem}\label{App-Theorem-A.2}{\rm
		{Assume that $A$ is the generator of a C$_0-$semigroup of contractions $\left(e^{tA}\right)_{t\geq0}$  on a Hilbert space $H$. If $A$ has no pure imaginary eigenvalues and  $\sigma\left(A\right)\cap i\mathbb{R}$ is countable,
			where $\sigma\left(A\right)$ denotes the spectrum of $A$, then the $C_0$-semigroup $\left(e^{tA}\right)_{t\geq0}$  is strongly stable.}\xqed{$\square$}}
\end{Theorem}
\noindent Concerning the characterization of polynomial stability stability of a $C_0-$semigroup of contraction $\left(e^{tA}\right)_{t\geq 0}$ we rely on the following result due to Borichev and Tomilov \cite{Borichev01} (see also \cite{Batty01} and \cite{RaoLiu01})
\begin{Theorem}\label{bt}
	{\rm
		Assume that $A$ is the generator of a strongly continuous semigroup of contractions $\left(e^{tA}\right)_{t\geq0}$  on $\mathcal{H}$.   If   $ i\mathbb{R}\subset \rho(\mathcal{A})$, then for a fixed $\ell>0$ the following conditions are equivalent
		\begin{equation}\label{h1}
		\sup_{\lambda\in\mathbb{R}}\left\|\left(i\lambda I-\mathcal{A}\right)^{-1}\right\|_{\mathcal{L}\left(\mathcal{H}\right)}=O\left(|\lambda|^\ell\right),
		\end{equation}
		\begin{equation}\label{h2}
		\|e^{t\mathcal{A}}U_{0}\|^2_{\HH} \leq \frac{C}{t^{\frac{2}{\ell}}}\|U_0\|^2_{D(\AA)},\hspace{0.1cm}\forall t>0,\hspace{0.1cm} U_0\in D(\AA),\hspace{0.1cm} \text{for some}\hspace{0.1cm} C>0.
		\end{equation}\xqed{$\square$}}
\end{Theorem}


\begin{thebibliography}{10}

\bibitem{Wehbe08}
F.~Abdallah, M.~Ghader, and A.~Wehbe.
\newblock Stability results of a distributed problem involving {B}resse system
  with history and/or {C}attaneo law under fully {D}irichlet or mixed boundary
  conditions.
\newblock {\em Math. Methods Appl. Sci.}, 41(5):1876--1907, 2018.

\bibitem{doi:10.1002/mma.6070}
M.~Afilal, A.~Guesmia, A.~Soufyane, and M.~Zahri.
\newblock On the exponential and polynomial stability for a linear {B}resse
  system.
\newblock {\em Mathematical Methods in the Applied Sciences}, 43(5):2626--2645,
  2020.

\bibitem{akilbadawiwehbe2020stability}
M.~Akil, H.~Badawi, and A.~Wehbe.
\newblock Stability results of a singular local interaction
  elastic/viscoelastic coupled wave equations with time delay.
\newblock {\em arXiv e-prints}, page arXiv:2007.08316, July 2020.

\bibitem{Akil2020}
M.~Akil, Y.~Chitour, M.~Ghader, and A.~Wehbe.
\newblock Stability and exact controllability of a {T}imoshenko system with
  only one fractional damping on the boundary.
\newblock {\em Asymptotic Analysis}, 119:221--280, 2020.
\newblock 3-4.

\bibitem{akilissawehbe}
M.~{Akil}, I.~{Issa}, and A.~{Wehbe}.
\newblock {Stability results of an elastic/viscoelastic transmission problem of
  locally coupled waves with non smooth coefficients}.
\newblock {\em arXiv e-prints}, page arXiv:2004.06758, Apr. 2020.

\bibitem{ALABAUBOUSSOUIRA2011481}
F.~{Alabau Boussouira}, J.~E. {Muñoz Rivera}, and D.~da~S.~{Almeida Júnior}.
\newblock Stability to weak dissipative {B}resse system.
\newblock {\em Journal of Mathematical Analysis and Applications}, 374(2):481
  -- 498, 2011.

\bibitem{Alves2014}
M.~Alves, J.~M. Rivera, M.~Sep{\'{u}}lveda, and O.~V. Villagr{\'{a}}n.
\newblock The lack of exponential stability in certain transmission problems
  with localized {K}elvin--{V}oigt dissipation.
\newblock {\em {SIAM} Journal on Applied Mathematics}, 74(2):345--365, Jan.
  2014.

\bibitem{Alves2013}
M.~Alves, J.~M. Rivera, M.~Sep{\'{u}}lveda, O.~V. Villagr{\'{a}}n, and M.~Z.
  Garay.
\newblock The asymptotic behavior of the linear transmission problem in
  viscoelasticity.
\newblock {\em Mathematische Nachrichten}, 287(5-6):483--497, Oct. 2013.

\bibitem{Arendt01}
W.~Arendt and C.~J.~K. Batty.
\newblock Tauberian theorems and stability of one-parameter semigroups.
\newblock {\em Trans. Amer. Math. Soc.}, 306(2):837--852, 1988.

\bibitem{BASSAM20151177}
M.~Bassam, D.~Mercier, S.~Nicaise, and A.~Wehbe.
\newblock Polynomial stability of the {T}imoshenko system by one boundary
  damping.
\newblock {\em Journal of Mathematical Analysis and Applications}, 425(2):1177
  -- 1203, 2015.

\bibitem{doi:10.1002/zamm.201500172}
M.~Bassam, D.~Mercier, S.~Nicaise, and A.~Wehbe.
\newblock Stability results of some distributed systems involving
  mindlin-{T}imoshenko plates in the plane.
\newblock {\em ZAMM - Journal of Applied Mathematics and Mechanics /
  Zeitschrift für Angewandte Mathematik und Mechanik}, 96(8):916--938, 2016.

\bibitem{Batty01}
C.~J.~K. Batty and T.~Duyckaerts.
\newblock \href {http://dx.doi.org/10.1007/s00028-008-0424-1} {Non-uniform
  stability for bounded semi-groups on {B}anach spaces}.
\newblock {\em J. Evol. Equ.}, 8(4):765--780, 2008.

\bibitem{Borichev01}
A.~Borichev and Y.~Tomilov.
\newblock Optimal polynomial decay of functions and operator semigroups.
\newblock {\em Math. Ann.}, 347(2):455--478, 2010.

\bibitem{deLima2018}
P.~R. de~Lima and H.~D. Fern{\'a}ndez~Sare.
\newblock Stability of thermoelastic {B}resse systems.
\newblock {\em Zeitschrift f{\"u}r angewandte Mathematik und Physik}, 70(1):3,
  Nov 2018.

\bibitem{ElArwadi2019}
T.~El~Arwadi and W.~Youssef.
\newblock On the stabilization of the {B}resse beam with {K}elvin--{V}oigt
  damping.
\newblock {\em Applied Mathematics {\&} Optimization}, Sep 2019.

\bibitem{doi:10.1080/00036811.2018.1520982}
L.~H. Fatori, M.~de~Oliveira~Alves, and H.~D.~F. Sare.
\newblock Stability conditions to {B}resse systems with indefinite memory
  dissipation.
\newblock {\em Applicable Analysis}, 99(6):1066--1084, 2020.

\bibitem{FATORI2012600}
L.~H. Fatori and R.~N. Monteiro.
\newblock The optimal decay rate for a weak dissipative {B}resse system.
\newblock {\em Applied Mathematics Letters}, 25(3):600 -- 604, 2012.

\bibitem{10.1093/imamat/hxq038}
L.~H. Fatori and J.~E. Muñoz~Rivera.
\newblock {Rates of decay to weak thermoelastic {B}resse system}.
\newblock {\em IMA Journal of Applied Mathematics}, 75(6):881--904, 06 2010.

\bibitem{gerbi2020stabilization}
S.~Gerbi, C.~Kassem, and A.~Wehbe.
\newblock Stabilization of non-smooth transmission problem involving {B}resse
  systems.
\newblock {\em arXiv e-prints}, page arXiv:2006.16595, June 2020.

\bibitem{doi:10.1002/mma.6918}
M.~Ghader, R.~Nasser, and A.~Wehbe.
\newblock Optimal polynomial stability of a string with locally distributed
  {K}elvin–{V}oigt damping and nonsmooth coefficient at the interface.
\newblock {\em Mathematical Methods in the Applied Sciences}.
\newblock to appear.

\bibitem{ghader2020stability}
M.~Ghader, R.~Nasser, and A.~Wehbe.
\newblock Stability results for an elastic-viscoelastic waves interaction
  systems with localized {K}elvin-{V}oigt damping and with an internal or
  boundary time delay.
\newblock {\em arXiv e-prints}, page arXiv:2003.12967, Mar. 2020.

\bibitem{ghader2020transmission2}
M.~Ghader and A.~Wehbe.
\newblock A transmission problem for the {T}imoshenko system with one local
  {K}elvin--{V}oigt damping and non-smooth coefficient at the interface.
\newblock {\em arXiv e-prints}, page arXiv:2005.12756, May 2020.

\bibitem{Guesmia2017}
A.~Guesmia.
\newblock Asymptotic stability of {B}resse system with one infinite memory in
  the longitudinal displacements.
\newblock {\em Mediterranean Journal of Mathematics}, 14(2):49, Mar 2017.

\bibitem{doi:10.1002/mma.3228}
A.~Guesmia and M.~Kafini.
\newblock {B}resse system with infinite memories.
\newblock {\em Mathematical Methods in the Applied Sciences},
  38(11):2389--2402, 2015.

\bibitem{F.HASSINE2015}
F.~Hassine.
\newblock Stability of elastic transmission systems with a local
  {K}elvin--{V}oigt damping.
\newblock {\em European Journal of Control}, 23:84--93, May 2015.

\bibitem{HASSINE201584}
F.~Hassine.
\newblock Energy decay estimates of elastic transmission wave/beam systems with
  a local {K}elvin--{V}oigt damping.
\newblock {\em International Journal of Control}, 89(10):1933--1950, June 2016.

\bibitem{Hayek}
A.~Hayek, S.~Nicaise, Z.~Salloum, and A.~Wehbe.
\newblock A transmission problem of a system of weakly coupled wave equations
  with {K}elvin-{V}oigt dampings and non-smooth coefficient at the interface.
\newblock {\em SeMA J.}, 77(3):305--338, 2020.

\bibitem{Huang-1988}
F.~Huang.
\newblock On the mathematical model for linear elastic systems with analytic
  damping.
\newblock {\em {SIAM} Journal on Control and Optimization}, 26(3):714--724, may
  1988.

\bibitem{Kato01}
T.~Kato.
\newblock {\em Perturbation Theory for Linear Operators}.
\newblock Springer Berlin Heidelberg, 1995.

\bibitem{LagneseLeugering01}
J.~E. Lagnese, G.~Leugering, and E.~J. P.~G. Schmidt.
\newblock {\em Modeling, Analysis and Control of Dynamic Elastic Multi-Link
  Structures}.
\newblock Birkh\"{a}user Boston, 1994.

\bibitem{chenLiuLiu-1998}
K.~Liu, S.~Chen, and Z.~Liu.
\newblock \href{https://doi.org/10.1137/s0036139996292015} {Spectrum and
  Stability for Elastic Systems with Global or Local {K}elvin--{V}oigt
  Damping}.
\newblock {\em {SIAM} Journal on Applied Mathematics}, 59(2):651--668, jan
  1998.

\bibitem{RaoLiu01}
Z.~Liu and B.~Rao.
\newblock Characterization of polynomial decay rate for the solution of linear
  evolution equation.
\newblock {\em Z. Angew. Math. Phys.}, 56(4):630--644, 2005.

\bibitem{RaoLiu03}
Z.~Liu and B.~Rao.
\newblock \href {http://dx.doi.org/10.1007/s00033-008-6122-6}{Energy decay rate
  of the thermoelastic {B}resse system}.
\newblock {\em Z. Angew. Math. Phys.}, 60(1):54--69, 2009.

\bibitem{Liu2016}
Z.~Liu and Q.~Zhang.
\newblock \href {https://doi.org/10.1137/15m1049385} {Stability of a String
  with Local {K}elvin--{V}oigt Damping and Nonsmooth Coefficient at Interface}.
\newblock {\em {SIAM} Journal on Control and Optimization}, 54(4):1859--1871,
  Jan. 2016.

\bibitem{doi:10.1137/15M1049385}
Z.~Liu and Q.~Zhang.
\newblock Stability of a string with local {K}elvin--{V}oigtdamping and
  nonsmooth coefficient at interface.
\newblock {\em SIAM Journal on Control and Optimization}, 54(4):1859--1871,
  2016.

\bibitem{Wehbe03}
N.~Najdi and A.~Wehbe.
\newblock Weakly locally thermal stabilization of {B}resse systems.
\newblock {\em Electron. J. Differential Equations}, pages No. 182, 19, 2014.

\bibitem{NASSER2019272}
R.~Nasser, N.~Noun, and A.~Wehbe.
\newblock Stabilization of the wave equations with localized {K}elvin–{V}oigt
  type damping under optimal geometric conditions.
\newblock {\em Comptes Rendus Mathematique}, 357(3):272 -- 277, 2019.

\bibitem{Wehbe02}
N.~Noun and A.~Wehbe.
\newblock Stabilisation faible interne locale de syst\`eme \'elastique de
  {B}resse.
\newblock {\em C. R. Math. Acad. Sci. Paris}, 350(9-10):493--498, 2012.

\bibitem{Portillo2017}
H.~P. Oquendo.
\newblock \href{https://doi.org/10.1002/mma.4510} {Frictional versus
  {K}elvin--{V}oigt damping in a transmission problem}.
\newblock {\em Mathematical Methods in the Applied Sciences},
  40(18):7026--7032, July 2017.

\bibitem{Pazy01}
A.~Pazy.
\newblock {\em Semigroups of linear operators and applications to partial
  differential equations}, volume~44 of {\em Applied Mathematical Sciences}.
\newblock Springer-Verlag, New York, 1983.

\bibitem{rivera2018}
J.~E.~M. Rivera, O.~V. Villagran, and M.~Sepulveda.
\newblock \href{https://doi.org/10.1080/03605302.2018.1475490} {Stability to
  localized viscoelastic transmission problem}.
\newblock {\em Communications in Partial Differential Equations},
  43(5):821--838, May 2018.

\bibitem{doi:10.1080/00036810903156149}
A.~Wehbe and W.~Youssef.
\newblock Stabilization of the uniform {T}imoshenko beam by one locally
  distributed feedback.
\newblock {\em Applicable Analysis}, 88(7):1067--1078, 2009.

\bibitem{Wehbe01}
A.~Wehbe and W.~Youssef.
\newblock Exponential and polynomial stability of an elastic {B}resse system
  with two locally distributed feedbacks.
\newblock {\em J. Math. Phys.}, 51(10):103523, 17, 2010.

\end{thebibliography}

\end{document}